\documentclass[letterpaper,11pt]{article}
\usepackage{url}
\usepackage[nottoc,notlot,notlof]{tocbibind}
\usepackage{tabularx} 
\usepackage{amsmath,amsthm,amssymb}  
\usepackage{amsmath}
\usepackage{amsfonts}
\usepackage{amssymb}
\usepackage{graphicx} 
\usepackage[margin=0.8in,letterpaper]{geometry} 
\usepackage{cite} 
\usepackage[final]{hyperref} 
\usepackage[english]{babel}
\usepackage{caption}
\usepackage{float}
\usepackage[table]{xcolor}
\usepackage{array}
\usepackage{algorithm}
\usepackage{algpseudocode}
\usepackage{cases}
\usepackage{titlesec}
\usepackage{tikz}
\usepackage[utf8]{inputenc}
\usepackage[T1]{fontenc}

\usepackage{mathtools}
\usepackage{extarrows}
\usepackage{booktabs} 
\usepackage{dsfont}

\newtheorem{theorem}{Theorem}

\newtheorem{proposition}{Proposition}
\newtheorem{corollary}{Corollary}

\usepackage{microtype}      
\usepackage{lipsum}
\usepackage{nicefrac}
\usepackage{fancyhdr}       

\author{
 El-Mehdi Mehiri \\
  Department of Mathematics and Industrial Engineering,\\
  Polytechnique Montréal, Université de Montréal. \\
  \texttt{elmehdi.mehiri@polymtl.ca} \\
  \texttt{mehiri314@gmail.com} \\
}
\title{\textbf{On the restricted Hanoi Graphs}}

\date{}

\begin{document}

\maketitle

 \begin{abstract}
 \noindent
Consider the restricted Hanoi graphs which correspond to the variants of the famous Tower of Hanoi problem with multiple pegs where moves of the discs are restricted throughout the arcs of a movement digraph whose vertices represent the pegs of the puzzle and an arc from vertex $p$ to vertex $q$ exists if and only if moves from peg $p$ to peg $q$ are allowed. In this paper, we gave some notes on how to construct the restricted Hanoi graphs as well as some combinatorial results on the number of arcs in these graphs.
 \\[2mm]
 \end{abstract}

\baselineskip=0.20in

\section{Introduction}
The well-known Tower of Hanoi problem was first proposed by the french mathematician Edouard Lucas in 1884 \cite{itard1960ed}. In this puzzle, three pegs and $n$ discs are given of distinct diameters, initially, all discs are stacked on one peg called the source peg in which they are ordered monotonically by diameter with the smallest at the top and the largest at the bottom. The player is required to transfer the tower of the $n$ discs from the source peg to another peg called the destination peg using the minimum number of moves, and respecting the following rules:

\begin{itemize}
\item[$(i)$] at each step only one disc can be moved;
       \item[$(ii)$] the disc moved must be a topmost disc;
       \item[$(iii)$] a disc cannot reside on a smaller one.
\end{itemize}
The unique optimal solution to this problem can be provided by the well-known recursive algorithm that accomplishes this task using
$2^{n}+1$ moves \cite{hinz2018tower}.

Many variants of the Tower of Hanoi problem have been considered in the literature, such as increasing the number
of pegs \cite{dudeney2002canterbury,frame1941solution,stockmeyer1994variations }, forbidding certain movements
of discs between certain pegs \cite{stockmeyer1994variations, atkinson1981cyclic, sapir2004tower}, and allowing the discs to be placed on top of smaller discs \cite{wood1980towers}. Composing these variants is for the sake of studying their
combinatorial as well as algorithmic aspects. For more about the Tower of Hanoi
problem, we point the interested reader to the comprehensive monograph \cite{hinz2018tower}.

In this paper, we consider the generalized Tower of Hanoi in which the number of pegs is  arbitrary, and moves between pegs are restricted to certain rules. 

Let $\mathcal{N}=\{1,\ldots,n\}$ and $\mathcal{M}=\{1,\ldots,m\}$ be the sets of discs and pegs respectively where $n\geq 0$ and $m\geq 3$. 
The allowed moves between pegs in a variant of the Tower of Hanoi based on restricting disc moves can be characterized using a digraph $D=(\mathcal{M},A(D))$ such that $(p,q)\in A(D)$ if and only if the movement of a disc from peg $p$ to peg $q$ is allowed. The state graph of the Tower of Hanoi of $n$ discs and $m$ pegs in which moves between pegs are restricted using a movement digraph $D$ is denoted by $H_{n}^{m}(D)$, where the set of vertices of $H_{n}^{m}(D)$ is the set of all possible distributions of the $n$ discs on the $m$ pegs, and two states (vertices) being adjacent whenever one is obtained from the other by a
legal move.  Each vertex $u\in V(H_{n}^{m}(D))$ is represented  as an $n$-tuple, i.e,  $u = u_{n}\ldots u_{1}$ where $u_{d}\in \mathcal{M}$ is the peg on which disc $d\in\mathcal{N}$ is lying in state $u$ \cite{hinz2022dudeney}. 

It is proven that the only solvable variants based on disc movement restrictions are those having a strongly connected movement digraph \cite[Proposition 8.4]{hinz2018tower}, i.e., the task to transfer a  tower of $n$ discs from one peg to another has always a  solution, no matter the source and destination pegs and how many discs are considered. Among the most known and studied variants of the Tower of Hanoi with restricted disc moves on $m$ pegs, we cite the classical, the linear, the cyclic, and the star variants having the movement digraph $\overleftrightarrow{K}_{m}$, $\overleftrightarrow{P}_{m}$, $\overrightarrow{C}_{m}$, and $\overleftrightarrow{K}_{1,m-1}$ respectively. For more about these variants, we refer to \cite{hinz2022dudeney,sapir2004tower,stockmeyer1994variations,atkinson1981cyclic,berend2006cyclic, allouche2005restricted,berend2012tower}.
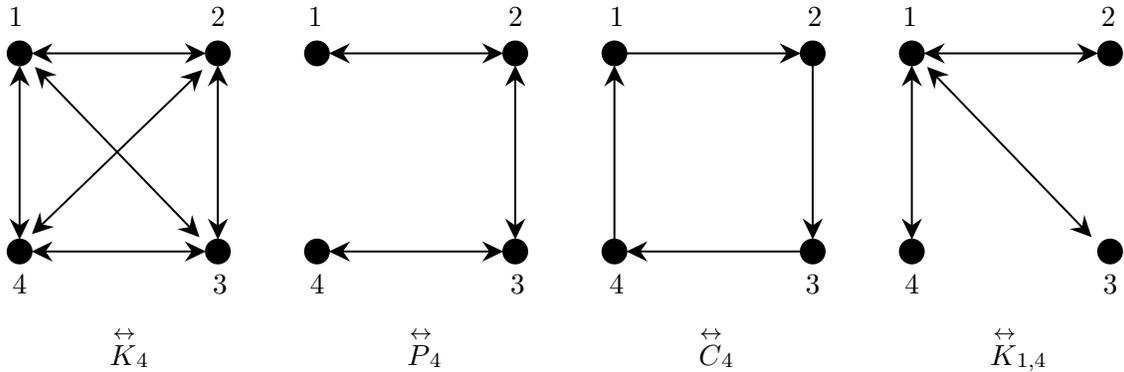
\begin{figure}[H]
    \centering

\tikzset{every picture/.style={line width=0.75pt}} 

\begin{tikzpicture}[x=0.75pt,y=0.75pt,yscale=-1,xscale=1]

\draw  [fill={rgb, 255:red, 0; green, 0; blue, 0 }  ,fill opacity=1 ] (166,51) .. controls (166,47.69) and (168.69,45) .. (172,45) .. controls (175.31,45) and (178,47.69) .. (178,51) .. controls (178,54.31) and (175.31,57) .. (172,57) .. controls (168.69,57) and (166,54.31) .. (166,51) -- cycle ;
\draw  [fill={rgb, 255:red, 0; green, 0; blue, 0 }  ,fill opacity=1 ] (66,51) .. controls (66,47.69) and (68.69,45) .. (72,45) .. controls (75.31,45) and (78,47.69) .. (78,51) .. controls (78,54.31) and (75.31,57) .. (72,57) .. controls (68.69,57) and (66,54.31) .. (66,51) -- cycle ;
\draw  [fill={rgb, 255:red, 0; green, 0; blue, 0 }  ,fill opacity=1 ] (166,151) .. controls (166,147.69) and (168.69,145) .. (172,145) .. controls (175.31,145) and (178,147.69) .. (178,151) .. controls (178,154.31) and (175.31,157) .. (172,157) .. controls (168.69,157) and (166,154.31) .. (166,151) -- cycle ;
\draw  [fill={rgb, 255:red, 0; green, 0; blue, 0 }  ,fill opacity=1 ] (66,151) .. controls (66,147.69) and (68.69,145) .. (72,145) .. controls (75.31,145) and (78,147.69) .. (78,151) .. controls (78,154.31) and (75.31,157) .. (72,157) .. controls (68.69,157) and (66,154.31) .. (66,151) -- cycle ;
\draw    (81,51) -- (163,51) ;
\draw [shift={(166,51)}, rotate = 180] [fill={rgb, 255:red, 0; green, 0; blue, 0 }  ][line width=0.08]  [draw opacity=0] (10.72,-5.15) -- (0,0) -- (10.72,5.15) -- (7.12,0) -- cycle    ;
\draw [shift={(78,51)}, rotate = 0] [fill={rgb, 255:red, 0; green, 0; blue, 0 }  ][line width=0.08]  [draw opacity=0] (10.72,-5.15) -- (0,0) -- (10.72,5.15) -- (7.12,0) -- cycle    ;
\draw    (81,151) -- (163,151) ;
\draw [shift={(166,151)}, rotate = 180] [fill={rgb, 255:red, 0; green, 0; blue, 0 }  ][line width=0.08]  [draw opacity=0] (10.72,-5.15) -- (0,0) -- (10.72,5.15) -- (7.12,0) -- cycle    ;
\draw [shift={(78,151)}, rotate = 0] [fill={rgb, 255:red, 0; green, 0; blue, 0 }  ][line width=0.08]  [draw opacity=0] (10.72,-5.15) -- (0,0) -- (10.72,5.15) -- (7.12,0) -- cycle    ;
\draw    (172,142) -- (172,60) ;
\draw [shift={(172,57)}, rotate = 90] [fill={rgb, 255:red, 0; green, 0; blue, 0 }  ][line width=0.08]  [draw opacity=0] (10.72,-5.15) -- (0,0) -- (10.72,5.15) -- (7.12,0) -- cycle    ;
\draw [shift={(172,145)}, rotate = 270] [fill={rgb, 255:red, 0; green, 0; blue, 0 }  ][line width=0.08]  [draw opacity=0] (10.72,-5.15) -- (0,0) -- (10.72,5.15) -- (7.12,0) -- cycle    ;
\draw    (72,142) -- (72,60) ;
\draw [shift={(72,57)}, rotate = 90] [fill={rgb, 255:red, 0; green, 0; blue, 0 }  ][line width=0.08]  [draw opacity=0] (10.72,-5.15) -- (0,0) -- (10.72,5.15) -- (7.12,0) -- cycle    ;
\draw [shift={(72,145)}, rotate = 270] [fill={rgb, 255:red, 0; green, 0; blue, 0 }  ][line width=0.08]  [draw opacity=0] (10.72,-5.15) -- (0,0) -- (10.72,5.15) -- (7.12,0) -- cycle    ;
\draw    (80.36,140.22) -- (162.04,61.38) ;
\draw [shift={(164.2,59.3)}, rotate = 136.02] [fill={rgb, 255:red, 0; green, 0; blue, 0 }  ][line width=0.08]  [draw opacity=0] (10.72,-5.15) -- (0,0) -- (10.72,5.15) -- (7.12,0) -- cycle    ;
\draw [shift={(78.2,142.3)}, rotate = 316.02] [fill={rgb, 255:red, 0; green, 0; blue, 0 }  ][line width=0.08]  [draw opacity=0] (10.72,-5.15) -- (0,0) -- (10.72,5.15) -- (7.12,0) -- cycle    ;
\draw    (81.78,59.96) -- (160.62,141.64) ;
\draw [shift={(162.7,143.8)}, rotate = 226.02] [fill={rgb, 255:red, 0; green, 0; blue, 0 }  ][line width=0.08]  [draw opacity=0] (10.72,-5.15) -- (0,0) -- (10.72,5.15) -- (7.12,0) -- cycle    ;
\draw [shift={(79.7,57.8)}, rotate = 46.02] [fill={rgb, 255:red, 0; green, 0; blue, 0 }  ][line width=0.08]  [draw opacity=0] (10.72,-5.15) -- (0,0) -- (10.72,5.15) -- (7.12,0) -- cycle    ;
\draw  [fill={rgb, 255:red, 0; green, 0; blue, 0 }  ,fill opacity=1 ] (316,51) .. controls (316,47.69) and (318.69,45) .. (322,45) .. controls (325.31,45) and (328,47.69) .. (328,51) .. controls (328,54.31) and (325.31,57) .. (322,57) .. controls (318.69,57) and (316,54.31) .. (316,51) -- cycle ;
\draw  [fill={rgb, 255:red, 0; green, 0; blue, 0 }  ,fill opacity=1 ] (216,51) .. controls (216,47.69) and (218.69,45) .. (222,45) .. controls (225.31,45) and (228,47.69) .. (228,51) .. controls (228,54.31) and (225.31,57) .. (222,57) .. controls (218.69,57) and (216,54.31) .. (216,51) -- cycle ;
\draw  [fill={rgb, 255:red, 0; green, 0; blue, 0 }  ,fill opacity=1 ] (316,151) .. controls (316,147.69) and (318.69,145) .. (322,145) .. controls (325.31,145) and (328,147.69) .. (328,151) .. controls (328,154.31) and (325.31,157) .. (322,157) .. controls (318.69,157) and (316,154.31) .. (316,151) -- cycle ;
\draw  [fill={rgb, 255:red, 0; green, 0; blue, 0 }  ,fill opacity=1 ] (216,151) .. controls (216,147.69) and (218.69,145) .. (222,145) .. controls (225.31,145) and (228,147.69) .. (228,151) .. controls (228,154.31) and (225.31,157) .. (222,157) .. controls (218.69,157) and (216,154.31) .. (216,151) -- cycle ;
\draw    (231,51) -- (313,51) ;
\draw [shift={(316,51)}, rotate = 180] [fill={rgb, 255:red, 0; green, 0; blue, 0 }  ][line width=0.08]  [draw opacity=0] (10.72,-5.15) -- (0,0) -- (10.72,5.15) -- (7.12,0) -- cycle    ;
\draw [shift={(228,51)}, rotate = 0] [fill={rgb, 255:red, 0; green, 0; blue, 0 }  ][line width=0.08]  [draw opacity=0] (10.72,-5.15) -- (0,0) -- (10.72,5.15) -- (7.12,0) -- cycle    ;
\draw    (231,151) -- (313,151) ;
\draw [shift={(316,151)}, rotate = 180] [fill={rgb, 255:red, 0; green, 0; blue, 0 }  ][line width=0.08]  [draw opacity=0] (10.72,-5.15) -- (0,0) -- (10.72,5.15) -- (7.12,0) -- cycle    ;
\draw [shift={(228,151)}, rotate = 0] [fill={rgb, 255:red, 0; green, 0; blue, 0 }  ][line width=0.08]  [draw opacity=0] (10.72,-5.15) -- (0,0) -- (10.72,5.15) -- (7.12,0) -- cycle    ;
\draw    (322,142) -- (322,60) ;
\draw [shift={(322,57)}, rotate = 90] [fill={rgb, 255:red, 0; green, 0; blue, 0 }  ][line width=0.08]  [draw opacity=0] (10.72,-5.15) -- (0,0) -- (10.72,5.15) -- (7.12,0) -- cycle    ;
\draw [shift={(322,145)}, rotate = 270] [fill={rgb, 255:red, 0; green, 0; blue, 0 }  ][line width=0.08]  [draw opacity=0] (10.72,-5.15) -- (0,0) -- (10.72,5.15) -- (7.12,0) -- cycle    ;
\draw  [fill={rgb, 255:red, 0; green, 0; blue, 0 }  ,fill opacity=1 ] (466,51) .. controls (466,47.69) and (468.69,45) .. (472,45) .. controls (475.31,45) and (478,47.69) .. (478,51) .. controls (478,54.31) and (475.31,57) .. (472,57) .. controls (468.69,57) and (466,54.31) .. (466,51) -- cycle ;
\draw  [fill={rgb, 255:red, 0; green, 0; blue, 0 }  ,fill opacity=1 ] (366,51) .. controls (366,47.69) and (368.69,45) .. (372,45) .. controls (375.31,45) and (378,47.69) .. (378,51) .. controls (378,54.31) and (375.31,57) .. (372,57) .. controls (368.69,57) and (366,54.31) .. (366,51) -- cycle ;
\draw  [fill={rgb, 255:red, 0; green, 0; blue, 0 }  ,fill opacity=1 ] (466,151) .. controls (466,147.69) and (468.69,145) .. (472,145) .. controls (475.31,145) and (478,147.69) .. (478,151) .. controls (478,154.31) and (475.31,157) .. (472,157) .. controls (468.69,157) and (466,154.31) .. (466,151) -- cycle ;
\draw  [fill={rgb, 255:red, 0; green, 0; blue, 0 }  ,fill opacity=1 ] (366,151) .. controls (366,147.69) and (368.69,145) .. (372,145) .. controls (375.31,145) and (378,147.69) .. (378,151) .. controls (378,154.31) and (375.31,157) .. (372,157) .. controls (368.69,157) and (366,154.31) .. (366,151) -- cycle ;
\draw    (378,51) -- (463,51) ;
\draw [shift={(466,51)}, rotate = 180] [fill={rgb, 255:red, 0; green, 0; blue, 0 }  ][line width=0.08]  [draw opacity=0] (10.72,-5.15) -- (0,0) -- (10.72,5.15) -- (7.12,0) -- cycle    ;
\draw    (381,151) -- (466,151) ;
\draw [shift={(378,151)}, rotate = 0] [fill={rgb, 255:red, 0; green, 0; blue, 0 }  ][line width=0.08]  [draw opacity=0] (10.72,-5.15) -- (0,0) -- (10.72,5.15) -- (7.12,0) -- cycle    ;
\draw    (472,142) -- (472,57) ;
\draw [shift={(472,145)}, rotate = 270] [fill={rgb, 255:red, 0; green, 0; blue, 0 }  ][line width=0.08]  [draw opacity=0] (10.72,-5.15) -- (0,0) -- (10.72,5.15) -- (7.12,0) -- cycle    ;
\draw    (372,145) -- (372,60) ;
\draw [shift={(372,57)}, rotate = 90] [fill={rgb, 255:red, 0; green, 0; blue, 0 }  ][line width=0.08]  [draw opacity=0] (10.72,-5.15) -- (0,0) -- (10.72,5.15) -- (7.12,0) -- cycle    ;
\draw  [fill={rgb, 255:red, 0; green, 0; blue, 0 }  ,fill opacity=1 ] (616,51) .. controls (616,47.69) and (618.69,45) .. (622,45) .. controls (625.31,45) and (628,47.69) .. (628,51) .. controls (628,54.31) and (625.31,57) .. (622,57) .. controls (618.69,57) and (616,54.31) .. (616,51) -- cycle ;
\draw  [fill={rgb, 255:red, 0; green, 0; blue, 0 }  ,fill opacity=1 ] (516,51) .. controls (516,47.69) and (518.69,45) .. (522,45) .. controls (525.31,45) and (528,47.69) .. (528,51) .. controls (528,54.31) and (525.31,57) .. (522,57) .. controls (518.69,57) and (516,54.31) .. (516,51) -- cycle ;
\draw  [fill={rgb, 255:red, 0; green, 0; blue, 0 }  ,fill opacity=1 ] (616,151) .. controls (616,147.69) and (618.69,145) .. (622,145) .. controls (625.31,145) and (628,147.69) .. (628,151) .. controls (628,154.31) and (625.31,157) .. (622,157) .. controls (618.69,157) and (616,154.31) .. (616,151) -- cycle ;
\draw  [fill={rgb, 255:red, 0; green, 0; blue, 0 }  ,fill opacity=1 ] (516,151) .. controls (516,147.69) and (518.69,145) .. (522,145) .. controls (525.31,145) and (528,147.69) .. (528,151) .. controls (528,154.31) and (525.31,157) .. (522,157) .. controls (518.69,157) and (516,154.31) .. (516,151) -- cycle ;
\draw    (531,51) -- (613,51) ;
\draw [shift={(616,51)}, rotate = 180] [fill={rgb, 255:red, 0; green, 0; blue, 0 }  ][line width=0.08]  [draw opacity=0] (10.72,-5.15) -- (0,0) -- (10.72,5.15) -- (7.12,0) -- cycle    ;
\draw [shift={(528,51)}, rotate = 0] [fill={rgb, 255:red, 0; green, 0; blue, 0 }  ][line width=0.08]  [draw opacity=0] (10.72,-5.15) -- (0,0) -- (10.72,5.15) -- (7.12,0) -- cycle    ;
\draw    (522,142) -- (522,60) ;
\draw [shift={(522,57)}, rotate = 90] [fill={rgb, 255:red, 0; green, 0; blue, 0 }  ][line width=0.08]  [draw opacity=0] (10.72,-5.15) -- (0,0) -- (10.72,5.15) -- (7.12,0) -- cycle    ;
\draw [shift={(522,145)}, rotate = 270] [fill={rgb, 255:red, 0; green, 0; blue, 0 }  ][line width=0.08]  [draw opacity=0] (10.72,-5.15) -- (0,0) -- (10.72,5.15) -- (7.12,0) -- cycle    ;
\draw    (531.78,59.96) -- (610.62,141.64) ;
\draw [shift={(612.7,143.8)}, rotate = 226.02] [fill={rgb, 255:red, 0; green, 0; blue, 0 }  ][line width=0.08]  [draw opacity=0] (10.72,-5.15) -- (0,0) -- (10.72,5.15) -- (7.12,0) -- cycle    ;
\draw [shift={(529.7,57.8)}, rotate = 46.02] [fill={rgb, 255:red, 0; green, 0; blue, 0 }  ][line width=0.08]  [draw opacity=0] (10.72,-5.15) -- (0,0) -- (10.72,5.15) -- (7.12,0) -- cycle    ;

\draw (116,188.4) node [anchor=north west][inner sep=0.75pt]    {$\stackrel {\leftrightarrow}{K}_{4}$};
\draw (65,26.4) node [anchor=north west][inner sep=0.75pt]    {$1$};
\draw (216,26.4) node [anchor=north west][inner sep=0.75pt]    {$1$};
\draw (368,26.4) node [anchor=north west][inner sep=0.75pt]    {$1$};
\draw (516,26.4) node [anchor=north west][inner sep=0.75pt]    {$1$};
\draw (167,26.4) node [anchor=north west][inner sep=0.75pt]    {$2$};
\draw (317,26.4) node [anchor=north west][inner sep=0.75pt]    {$2$};
\draw (468,26.4) node [anchor=north west][inner sep=0.75pt]    {$2$};
\draw (616,26.4) node [anchor=north west][inner sep=0.75pt]    {$2$};
\draw (168,161.4) node [anchor=north west][inner sep=0.75pt]    {$3$};
\draw (318,161.4) node [anchor=north west][inner sep=0.75pt]    {$3$};
\draw (467,161.4) node [anchor=north west][inner sep=0.75pt]    {$3$};
\draw (617,161.4) node [anchor=north west][inner sep=0.75pt]    {$3$};
\draw (517,161.4) node [anchor=north west][inner sep=0.75pt]    {$4$};
\draw (368,161.4) node [anchor=north west][inner sep=0.75pt]    {$4$};
\draw (217,161.4) node [anchor=north west][inner sep=0.75pt]    {$4$};
\draw (67,161.4) node [anchor=north west][inner sep=0.75pt]    {$4$};
\draw (266,188.4) node [anchor=north west][inner sep=0.75pt]    {$\stackrel {\leftrightarrow}{P}_{4}$};
\draw (413,188.4) node [anchor=north west][inner sep=0.75pt]    {$\stackrel {\leftrightarrow}{C}_{4}$};
\draw (560,188.4) node [anchor=north west][inner sep=0.75pt]    {$\stackrel {\leftrightarrow}{K}_{1,4}$};

\end{tikzpicture}

    \caption{The movement digraph of the classical, linear, cyclic, and star variants with $m=4$ pegs.}
    \label{fig:my_label}
\end{figure}

The restricted Hanoi graphs $H_{n}^{m}(D)$, especially where $D=\overleftrightarrow{K}_{m}$ (which is the only variant of this type that has no restrictions on moves, means that all movements are allowed) were studied under different aspects, such as combinatorics, graphs theory, and algorithmic aspects \cite{berend2006diameter,klavvzar2005hanoi,azriel2008infinite,leiss2013worst}. 

In this paper, we study these graphs for any movement digraph $D$ from the combinatorial counting, graphs theory, and algorithmic points of view. More precisely, we present an easy-to-implement generator of the graphs $H_{n}^{m}(D)$ based on the characterization of its arcs. We also present the number of states that allows disc $k\in\mathcal{N}$ to be moved, which allows us to deduce a new more general proof of the number of arcs in $H_{n}^{m}(D)$, namely $|A(H_{n}^{m}(D))|$.

\begin{figure}[H]
    \centering

\tikzset{every picture/.style={line width=0.75pt}} 

\begin{tikzpicture}[scale=0.8,x=0.75pt,y=0.75pt,yscale=-1,xscale=1]

\draw  [fill={rgb, 255:red, 0; green, 0; blue, 0 }  ,fill opacity=1 ] (151.74,93.94) .. controls (151.74,90.84) and (154.22,88.33) .. (157.28,88.33) .. controls (160.34,88.33) and (162.82,90.84) .. (162.82,93.94) .. controls (162.82,97.03) and (160.34,99.54) .. (157.28,99.54) .. controls (154.22,99.54) and (151.74,97.03) .. (151.74,93.94) -- cycle ;
\draw  [fill={rgb, 255:red, 0; green, 0; blue, 0 }  ,fill opacity=1 ] (151.74,37.87) .. controls (151.74,34.77) and (154.22,32.26) .. (157.28,32.26) .. controls (160.34,32.26) and (162.82,34.77) .. (162.82,37.87) .. controls (162.82,40.97) and (160.34,43.48) .. (157.28,43.48) .. controls (154.22,43.48) and (151.74,40.97) .. (151.74,37.87) -- cycle ;
\draw    (189.58,150) -- (157.28,37.87) ;
\draw    (157.28,93.94) -- (157.28,37.87) ;
\draw  [fill={rgb, 255:red, 0; green, 0; blue, 0 }  ,fill opacity=1 ] (184.04,150) .. controls (184.04,146.91) and (186.52,144.39) .. (189.58,144.39) .. controls (192.64,144.39) and (195.12,146.91) .. (195.12,150) .. controls (195.12,153.1) and (192.64,155.61) .. (189.58,155.61) .. controls (186.52,155.61) and (184.04,153.1) .. (184.04,150) -- cycle ;
\draw  [fill={rgb, 255:red, 0; green, 0; blue, 0 }  ,fill opacity=1 ] (119.45,150) .. controls (119.45,146.91) and (121.93,144.39) .. (124.98,144.39) .. controls (128.04,144.39) and (130.52,146.91) .. (130.52,150) .. controls (130.52,153.1) and (128.04,155.61) .. (124.98,155.61) .. controls (121.93,155.61) and (119.45,153.1) .. (119.45,150) -- cycle ;
\draw  [fill={rgb, 255:red, 0; green, 0; blue, 0 }  ,fill opacity=1 ] (119.45,206.07) .. controls (119.45,202.97) and (121.93,200.46) .. (124.98,200.46) .. controls (128.04,200.46) and (130.52,202.97) .. (130.52,206.07) .. controls (130.52,209.16) and (128.04,211.67) .. (124.98,211.67) .. controls (121.93,211.67) and (119.45,209.16) .. (119.45,206.07) -- cycle ;
\draw  [fill={rgb, 255:red, 0; green, 0; blue, 0 }  ,fill opacity=1 ] (184.04,206.07) .. controls (184.04,202.97) and (186.52,200.46) .. (189.58,200.46) .. controls (192.64,200.46) and (195.12,202.97) .. (195.12,206.07) .. controls (195.12,209.16) and (192.64,211.67) .. (189.58,211.67) .. controls (186.52,211.67) and (184.04,209.16) .. (184.04,206.07) -- cycle ;
\draw  [fill={rgb, 255:red, 0; green, 0; blue, 0 }  ,fill opacity=1 ] (151.74,262.72) .. controls (151.74,259.62) and (154.22,257.11) .. (157.28,257.11) .. controls (160.34,257.11) and (162.82,259.62) .. (162.82,262.72) .. controls (162.82,265.81) and (160.34,268.32) .. (157.28,268.32) .. controls (154.22,268.32) and (151.74,265.81) .. (151.74,262.72) -- cycle ;
\draw  [fill={rgb, 255:red, 0; green, 0; blue, 0 }  ,fill opacity=1 ] (231.99,234.1) .. controls (231.99,231) and (234.47,228.49) .. (237.53,228.49) .. controls (240.59,228.49) and (243.07,231) .. (243.07,234.1) .. controls (243.07,237.2) and (240.59,239.71) .. (237.53,239.71) .. controls (234.47,239.71) and (231.99,237.2) .. (231.99,234.1) -- cycle ;
\draw  [fill={rgb, 255:red, 0; green, 0; blue, 0 }  ,fill opacity=1 ] (264.29,290.75) .. controls (264.29,287.65) and (266.77,285.14) .. (269.83,285.14) .. controls (272.88,285.14) and (275.36,287.65) .. (275.36,290.75) .. controls (275.36,293.84) and (272.88,296.35) .. (269.83,296.35) .. controls (266.77,296.35) and (264.29,293.84) .. (264.29,290.75) -- cycle ;
\draw  [fill={rgb, 255:red, 0; green, 0; blue, 0 }  ,fill opacity=1 ] (199.69,290.75) .. controls (199.69,287.65) and (202.17,285.14) .. (205.23,285.14) .. controls (208.29,285.14) and (210.77,287.65) .. (210.77,290.75) .. controls (210.77,293.84) and (208.29,296.35) .. (205.23,296.35) .. controls (202.17,296.35) and (199.69,293.84) .. (199.69,290.75) -- cycle ;
\draw  [fill={rgb, 255:red, 0; green, 0; blue, 0 }  ,fill opacity=1 ] (71.5,234.1) .. controls (71.5,231) and (73.98,228.49) .. (77.03,228.49) .. controls (80.09,228.49) and (82.57,231) .. (82.57,234.1) .. controls (82.57,237.2) and (80.09,239.71) .. (77.03,239.71) .. controls (73.98,239.71) and (71.5,237.2) .. (71.5,234.1) -- cycle ;
\draw  [fill={rgb, 255:red, 0; green, 0; blue, 0 }  ,fill opacity=1 ] (103.8,290.17) .. controls (103.8,287.07) and (106.27,284.56) .. (109.33,284.56) .. controls (112.39,284.56) and (114.87,287.07) .. (114.87,290.17) .. controls (114.87,293.26) and (112.39,295.77) .. (109.33,295.77) .. controls (106.27,295.77) and (103.8,293.26) .. (103.8,290.17) -- cycle ;
\draw  [fill={rgb, 255:red, 0; green, 0; blue, 0 }  ,fill opacity=1 ] (39.2,290.17) .. controls (39.2,287.07) and (41.68,284.56) .. (44.74,284.56) .. controls (47.79,284.56) and (50.27,287.07) .. (50.27,290.17) .. controls (50.27,293.26) and (47.79,295.77) .. (44.74,295.77) .. controls (41.68,295.77) and (39.2,293.26) .. (39.2,290.17) -- cycle ;
\draw  [fill={rgb, 255:red, 0; green, 0; blue, 0 }  ,fill opacity=1 ] (312.24,318.78) .. controls (312.24,315.68) and (314.72,313.17) .. (317.78,313.17) .. controls (320.83,313.17) and (323.31,315.68) .. (323.31,318.78) .. controls (323.31,321.88) and (320.83,324.39) .. (317.78,324.39) .. controls (314.72,324.39) and (312.24,321.88) .. (312.24,318.78) -- cycle ;
\draw  [fill={rgb, 255:red, 0; green, 0; blue, 0 }  ,fill opacity=1 ] (-8.75,318.2) .. controls (-8.75,315.1) and (-6.27,312.59) .. (-3.21,312.59) .. controls (-0.15,312.59) and (2.32,315.1) .. (2.32,318.2) .. controls (2.32,321.3) and (-0.15,323.81) .. (-3.21,323.81) .. controls (-6.27,323.81) and (-8.75,321.3) .. (-8.75,318.2) -- cycle ;
\draw  [fill={rgb, 255:red, 0; green, 0; blue, 0 }  ,fill opacity=1 ] (151.74,227.56) .. controls (151.74,224.46) and (154.22,221.95) .. (157.28,221.95) .. controls (160.34,221.95) and (162.82,224.46) .. (162.82,227.56) .. controls (162.82,230.66) and (160.34,233.17) .. (157.28,233.17) .. controls (154.22,233.17) and (151.74,230.66) .. (151.74,227.56) -- cycle ;
\draw    (124.98,150) -- (157.28,37.87) ;
\draw    (124.98,150) -- (157.28,93.94) ;
\draw    (189.58,150) -- (157.28,93.94) ;
\draw    (184.04,150) -- (130.52,150) ;
\draw    (124.98,200.46) -- (124.98,155.61) ;
\draw    (189.58,200.46) -- (189.58,155.61) ;
\draw    (184.04,206.07) -- (130.52,206.07) ;
\draw    (157.28,262.72) -- (124.98,206.07) ;
\draw    (157.28,262.72) -- (189.58,206.07) ;
\draw    (157.28,227.56) -- (124.98,206.07) ;
\draw    (157.28,257.11) -- (157.28,233.17) ;
\draw    (157.28,227.56) -- (189.58,206.07) ;
\draw    (80.91,231.05) -- (119.45,206.07) ;
\draw    (109.33,290.17) -- (157.28,262.72) ;
\draw    (189.58,206.07) -- (237.53,234.1) ;
\draw    (162.82,262.72) -- (201.8,286.45) ;
\draw    (109.33,290.17) -- (77.03,234.1) ;
\draw    (205.23,290.75) -- (237.53,234.1) ;
\draw    (44.74,290.17) -- (77.03,234.1) ;
\draw    (44.74,290.17) -- (109.33,290.17) ;
\draw    (-3.21,318.2) -- (109.33,290.17) ;
\draw    (-3.21,318.2) -- (77.03,234.1) ;
\draw    (-3.21,318.2) -- (44.74,290.17) ;
\draw    (269.83,290.75) -- (317.78,318.78) ;
\draw    (205.23,290.75) -- (317.78,318.78) ;
\draw    (237.53,234.1) -- (317.78,318.78) ;
\draw    (237.53,234.1) -- (269.83,290.75) ;
\draw    (205.23,290.75) -- (269.83,290.75) ;
\draw    (189.58,150) -- (237.53,234.1) ;
\draw    (125.48,150.29) -- (77.03,234.1) ;
\draw    (205.23,290.75) -- (109.33,290.17) ;
\draw    (317.78,318.78) -- (-3.21,318.2) ;
\draw    (317.78,318.78) -- (157.28,37.87) ;
\draw    (-3.21,318.2) -- (157.28,37.87) ;

\draw  [fill={rgb, 255:red, 0; green, 0; blue, 0 }  ,fill opacity=1 ] (403.46,77) .. controls (403.46,73.9) and (405.94,71.39) .. (409,71.39) .. controls (412.06,71.39) and (414.54,73.9) .. (414.54,77) .. controls (414.54,80.1) and (412.06,82.61) .. (409,82.61) .. controls (405.94,82.61) and (403.46,80.1) .. (403.46,77) -- cycle ;
\draw  [fill={rgb, 255:red, 0; green, 0; blue, 0 }  ,fill opacity=1 ] (403.46,147) .. controls (403.46,143.9) and (405.94,141.39) .. (409,141.39) .. controls (412.06,141.39) and (414.54,143.9) .. (414.54,147) .. controls (414.54,150.1) and (412.06,152.61) .. (409,152.61) .. controls (405.94,152.61) and (403.46,150.1) .. (403.46,147) -- cycle ;
\draw  [fill={rgb, 255:red, 0; green, 0; blue, 0 }  ,fill opacity=1 ] (403.46,217) .. controls (403.46,213.9) and (405.94,211.39) .. (409,211.39) .. controls (412.06,211.39) and (414.54,213.9) .. (414.54,217) .. controls (414.54,220.1) and (412.06,222.61) .. (409,222.61) .. controls (405.94,222.61) and (403.46,220.1) .. (403.46,217) -- cycle ;
\draw  [fill={rgb, 255:red, 0; green, 0; blue, 0 }  ,fill opacity=1 ] (403.46,287) .. controls (403.46,283.9) and (405.94,281.39) .. (409,281.39) .. controls (412.06,281.39) and (414.54,283.9) .. (414.54,287) .. controls (414.54,290.1) and (412.06,292.61) .. (409,292.61) .. controls (405.94,292.61) and (403.46,290.1) .. (403.46,287) -- cycle ;
\draw  [fill={rgb, 255:red, 0; green, 0; blue, 0 }  ,fill opacity=1 ] (473.46,77) .. controls (473.46,73.9) and (475.94,71.39) .. (479,71.39) .. controls (482.06,71.39) and (484.54,73.9) .. (484.54,77) .. controls (484.54,80.1) and (482.06,82.61) .. (479,82.61) .. controls (475.94,82.61) and (473.46,80.1) .. (473.46,77) -- cycle ;
\draw  [fill={rgb, 255:red, 0; green, 0; blue, 0 }  ,fill opacity=1 ] (473.46,147) .. controls (473.46,143.9) and (475.94,141.39) .. (479,141.39) .. controls (482.06,141.39) and (484.54,143.9) .. (484.54,147) .. controls (484.54,150.1) and (482.06,152.61) .. (479,152.61) .. controls (475.94,152.61) and (473.46,150.1) .. (473.46,147) -- cycle ;
\draw  [fill={rgb, 255:red, 0; green, 0; blue, 0 }  ,fill opacity=1 ] (473.46,217) .. controls (473.46,213.9) and (475.94,211.39) .. (479,211.39) .. controls (482.06,211.39) and (484.54,213.9) .. (484.54,217) .. controls (484.54,220.1) and (482.06,222.61) .. (479,222.61) .. controls (475.94,222.61) and (473.46,220.1) .. (473.46,217) -- cycle ;
\draw  [fill={rgb, 255:red, 0; green, 0; blue, 0 }  ,fill opacity=1 ] (473.46,287) .. controls (473.46,283.9) and (475.94,281.39) .. (479,281.39) .. controls (482.06,281.39) and (484.54,283.9) .. (484.54,287) .. controls (484.54,290.1) and (482.06,292.61) .. (479,292.61) .. controls (475.94,292.61) and (473.46,290.1) .. (473.46,287) -- cycle ;
\draw  [fill={rgb, 255:red, 0; green, 0; blue, 0 }  ,fill opacity=1 ] (543.46,77) .. controls (543.46,73.9) and (545.94,71.39) .. (549,71.39) .. controls (552.06,71.39) and (554.54,73.9) .. (554.54,77) .. controls (554.54,80.1) and (552.06,82.61) .. (549,82.61) .. controls (545.94,82.61) and (543.46,80.1) .. (543.46,77) -- cycle ;
\draw  [fill={rgb, 255:red, 0; green, 0; blue, 0 }  ,fill opacity=1 ] (543.46,147) .. controls (543.46,143.9) and (545.94,141.39) .. (549,141.39) .. controls (552.06,141.39) and (554.54,143.9) .. (554.54,147) .. controls (554.54,150.1) and (552.06,152.61) .. (549,152.61) .. controls (545.94,152.61) and (543.46,150.1) .. (543.46,147) -- cycle ;
\draw  [fill={rgb, 255:red, 0; green, 0; blue, 0 }  ,fill opacity=1 ] (543.46,217) .. controls (543.46,213.9) and (545.94,211.39) .. (549,211.39) .. controls (552.06,211.39) and (554.54,213.9) .. (554.54,217) .. controls (554.54,220.1) and (552.06,222.61) .. (549,222.61) .. controls (545.94,222.61) and (543.46,220.1) .. (543.46,217) -- cycle ;
\draw  [fill={rgb, 255:red, 0; green, 0; blue, 0 }  ,fill opacity=1 ] (543.46,287) .. controls (543.46,283.9) and (545.94,281.39) .. (549,281.39) .. controls (552.06,281.39) and (554.54,283.9) .. (554.54,287) .. controls (554.54,290.1) and (552.06,292.61) .. (549,292.61) .. controls (545.94,292.61) and (543.46,290.1) .. (543.46,287) -- cycle ;
\draw  [fill={rgb, 255:red, 0; green, 0; blue, 0 }  ,fill opacity=1 ] (613.46,287) .. controls (613.46,283.9) and (615.94,281.39) .. (619,281.39) .. controls (622.06,281.39) and (624.54,283.9) .. (624.54,287) .. controls (624.54,290.1) and (622.06,292.61) .. (619,292.61) .. controls (615.94,292.61) and (613.46,290.1) .. (613.46,287) -- cycle ;
\draw  [fill={rgb, 255:red, 0; green, 0; blue, 0 }  ,fill opacity=1 ] (613.46,217) .. controls (613.46,213.9) and (615.94,211.39) .. (619,211.39) .. controls (622.06,211.39) and (624.54,213.9) .. (624.54,217) .. controls (624.54,220.1) and (622.06,222.61) .. (619,222.61) .. controls (615.94,222.61) and (613.46,220.1) .. (613.46,217) -- cycle ;
\draw  [fill={rgb, 255:red, 0; green, 0; blue, 0 }  ,fill opacity=1 ] (613.46,147) .. controls (613.46,143.9) and (615.94,141.39) .. (619,141.39) .. controls (622.06,141.39) and (624.54,143.9) .. (624.54,147) .. controls (624.54,150.1) and (622.06,152.61) .. (619,152.61) .. controls (615.94,152.61) and (613.46,150.1) .. (613.46,147) -- cycle ;
\draw  [fill={rgb, 255:red, 0; green, 0; blue, 0 }  ,fill opacity=1 ] (613.46,77) .. controls (613.46,73.9) and (615.94,71.39) .. (619,71.39) .. controls (622.06,71.39) and (624.54,73.9) .. (624.54,77) .. controls (624.54,80.1) and (622.06,82.61) .. (619,82.61) .. controls (615.94,82.61) and (613.46,80.1) .. (613.46,77) -- cycle ;
\draw    (409,147) -- (409,77) ;
\draw    (409,217) -- (409,147) ;
\draw    (409,287) -- (409,217) ;
\draw    (479,217) -- (409,217) ;
\draw    (479,287) -- (409,287) ;
\draw    (479,287) -- (479,217) ;
\draw    (479,217) -- (479,147) ;
\draw    (479,147) -- (479,77) ;
\draw    (549,147) -- (549,77) ;
\draw    (549,217) -- (549,147) ;
\draw    (549,287) -- (549,217) ;
\draw    (619,287) -- (619,217) ;
\draw    (619,217) -- (619,147) ;
\draw    (619,147) -- (619,77) ;
\draw    (549,77) -- (479,77) ;
\draw    (549,287) -- (479,287) ;
\draw    (619,147) -- (549,147) ;
\draw    (619,77) -- (549,77) ;

\draw  [fill={rgb, 255:red, 0; green, 0; blue, 0 }  ,fill opacity=1 ] (403.46,450.86) .. controls (403.46,447.76) and (405.94,445.25) .. (409,445.25) .. controls (412.06,445.25) and (414.54,447.76) .. (414.54,450.86) .. controls (414.54,453.95) and (412.06,456.46) .. (409,456.46) .. controls (405.94,456.46) and (403.46,453.95) .. (403.46,450.86) -- cycle ;
\draw  [fill={rgb, 255:red, 0; green, 0; blue, 0 }  ,fill opacity=1 ] (403.46,520.86) .. controls (403.46,517.76) and (405.94,515.25) .. (409,515.25) .. controls (412.06,515.25) and (414.54,517.76) .. (414.54,520.86) .. controls (414.54,523.95) and (412.06,526.46) .. (409,526.46) .. controls (405.94,526.46) and (403.46,523.95) .. (403.46,520.86) -- cycle ;
\draw  [fill={rgb, 255:red, 0; green, 0; blue, 0 }  ,fill opacity=1 ] (403.46,590.86) .. controls (403.46,587.76) and (405.94,585.25) .. (409,585.25) .. controls (412.06,585.25) and (414.54,587.76) .. (414.54,590.86) .. controls (414.54,593.95) and (412.06,596.46) .. (409,596.46) .. controls (405.94,596.46) and (403.46,593.95) .. (403.46,590.86) -- cycle ;
\draw  [fill={rgb, 255:red, 0; green, 0; blue, 0 }  ,fill opacity=1 ] (403.46,660.86) .. controls (403.46,657.76) and (405.94,655.25) .. (409,655.25) .. controls (412.06,655.25) and (414.54,657.76) .. (414.54,660.86) .. controls (414.54,663.95) and (412.06,666.46) .. (409,666.46) .. controls (405.94,666.46) and (403.46,663.95) .. (403.46,660.86) -- cycle ;
\draw  [fill={rgb, 255:red, 0; green, 0; blue, 0 }  ,fill opacity=1 ] (473.46,450.86) .. controls (473.46,447.76) and (475.94,445.25) .. (479,445.25) .. controls (482.06,445.25) and (484.54,447.76) .. (484.54,450.86) .. controls (484.54,453.95) and (482.06,456.46) .. (479,456.46) .. controls (475.94,456.46) and (473.46,453.95) .. (473.46,450.86) -- cycle ;
\draw  [fill={rgb, 255:red, 0; green, 0; blue, 0 }  ,fill opacity=1 ] (473.46,520.86) .. controls (473.46,517.76) and (475.94,515.25) .. (479,515.25) .. controls (482.06,515.25) and (484.54,517.76) .. (484.54,520.86) .. controls (484.54,523.95) and (482.06,526.46) .. (479,526.46) .. controls (475.94,526.46) and (473.46,523.95) .. (473.46,520.86) -- cycle ;
\draw  [fill={rgb, 255:red, 0; green, 0; blue, 0 }  ,fill opacity=1 ] (473.46,590.86) .. controls (473.46,587.76) and (475.94,585.25) .. (479,585.25) .. controls (482.06,585.25) and (484.54,587.76) .. (484.54,590.86) .. controls (484.54,593.95) and (482.06,596.46) .. (479,596.46) .. controls (475.94,596.46) and (473.46,593.95) .. (473.46,590.86) -- cycle ;
\draw  [fill={rgb, 255:red, 0; green, 0; blue, 0 }  ,fill opacity=1 ] (473.46,660.86) .. controls (473.46,657.76) and (475.94,655.25) .. (479,655.25) .. controls (482.06,655.25) and (484.54,657.76) .. (484.54,660.86) .. controls (484.54,663.95) and (482.06,666.46) .. (479,666.46) .. controls (475.94,666.46) and (473.46,663.95) .. (473.46,660.86) -- cycle ;
\draw  [fill={rgb, 255:red, 0; green, 0; blue, 0 }  ,fill opacity=1 ] (543.46,450.86) .. controls (543.46,447.76) and (545.94,445.25) .. (549,445.25) .. controls (552.06,445.25) and (554.54,447.76) .. (554.54,450.86) .. controls (554.54,453.95) and (552.06,456.46) .. (549,456.46) .. controls (545.94,456.46) and (543.46,453.95) .. (543.46,450.86) -- cycle ;
\draw  [fill={rgb, 255:red, 0; green, 0; blue, 0 }  ,fill opacity=1 ] (543.46,520.86) .. controls (543.46,517.76) and (545.94,515.25) .. (549,515.25) .. controls (552.06,515.25) and (554.54,517.76) .. (554.54,520.86) .. controls (554.54,523.95) and (552.06,526.46) .. (549,526.46) .. controls (545.94,526.46) and (543.46,523.95) .. (543.46,520.86) -- cycle ;
\draw  [fill={rgb, 255:red, 0; green, 0; blue, 0 }  ,fill opacity=1 ] (543.46,590.86) .. controls (543.46,587.76) and (545.94,585.25) .. (549,585.25) .. controls (552.06,585.25) and (554.54,587.76) .. (554.54,590.86) .. controls (554.54,593.95) and (552.06,596.46) .. (549,596.46) .. controls (545.94,596.46) and (543.46,593.95) .. (543.46,590.86) -- cycle ;
\draw  [fill={rgb, 255:red, 0; green, 0; blue, 0 }  ,fill opacity=1 ] (543.46,660.86) .. controls (543.46,657.76) and (545.94,655.25) .. (549,655.25) .. controls (552.06,655.25) and (554.54,657.76) .. (554.54,660.86) .. controls (554.54,663.95) and (552.06,666.46) .. (549,666.46) .. controls (545.94,666.46) and (543.46,663.95) .. (543.46,660.86) -- cycle ;
\draw  [fill={rgb, 255:red, 0; green, 0; blue, 0 }  ,fill opacity=1 ] (613.46,660.86) .. controls (613.46,657.76) and (615.94,655.25) .. (619,655.25) .. controls (622.06,655.25) and (624.54,657.76) .. (624.54,660.86) .. controls (624.54,663.95) and (622.06,666.46) .. (619,666.46) .. controls (615.94,666.46) and (613.46,663.95) .. (613.46,660.86) -- cycle ;
\draw  [fill={rgb, 255:red, 0; green, 0; blue, 0 }  ,fill opacity=1 ] (613.46,590.86) .. controls (613.46,587.76) and (615.94,585.25) .. (619,585.25) .. controls (622.06,585.25) and (624.54,587.76) .. (624.54,590.86) .. controls (624.54,593.95) and (622.06,596.46) .. (619,596.46) .. controls (615.94,596.46) and (613.46,593.95) .. (613.46,590.86) -- cycle ;
\draw  [fill={rgb, 255:red, 0; green, 0; blue, 0 }  ,fill opacity=1 ] (613.46,520.86) .. controls (613.46,517.76) and (615.94,515.25) .. (619,515.25) .. controls (622.06,515.25) and (624.54,517.76) .. (624.54,520.86) .. controls (624.54,523.95) and (622.06,526.46) .. (619,526.46) .. controls (615.94,526.46) and (613.46,523.95) .. (613.46,520.86) -- cycle ;
\draw  [fill={rgb, 255:red, 0; green, 0; blue, 0 }  ,fill opacity=1 ] (613.46,450.86) .. controls (613.46,447.76) and (615.94,445.25) .. (619,445.25) .. controls (622.06,445.25) and (624.54,447.76) .. (624.54,450.86) .. controls (624.54,453.95) and (622.06,456.46) .. (619,456.46) .. controls (615.94,456.46) and (613.46,453.95) .. (613.46,450.86) -- cycle ;
\draw    (409,512.25) -- (409,450.86) ;
\draw [shift={(409,515.25)}, rotate = 270] [fill={rgb, 255:red, 0; green, 0; blue, 0 }  ][line width=0.08]  [draw opacity=0] (10.72,-5.15) -- (0,0) -- (10.72,5.15) -- (7.12,0) -- cycle    ;
\draw    (409,582.25) -- (409,520.86) ;
\draw [shift={(409,585.25)}, rotate = 270] [fill={rgb, 255:red, 0; green, 0; blue, 0 }  ][line width=0.08]  [draw opacity=0] (10.72,-5.15) -- (0,0) -- (10.72,5.15) -- (7.12,0) -- cycle    ;
\draw    (409,652.25) -- (409,590.86) ;
\draw [shift={(409,655.25)}, rotate = 270] [fill={rgb, 255:red, 0; green, 0; blue, 0 }  ][line width=0.08]  [draw opacity=0] (10.72,-5.15) -- (0,0) -- (10.72,5.15) -- (7.12,0) -- cycle    ;
\draw    (470.46,590.86) -- (409,590.86) ;
\draw [shift={(473.46,590.86)}, rotate = 180] [fill={rgb, 255:red, 0; green, 0; blue, 0 }  ][line width=0.08]  [draw opacity=0] (10.72,-5.15) -- (0,0) -- (10.72,5.15) -- (7.12,0) -- cycle    ;
\draw    (470.46,660.86) -- (409,660.86) ;
\draw [shift={(473.46,660.86)}, rotate = 180] [fill={rgb, 255:red, 0; green, 0; blue, 0 }  ][line width=0.08]  [draw opacity=0] (10.72,-5.15) -- (0,0) -- (10.72,5.15) -- (7.12,0) -- cycle    ;
\draw    (479,652.25) -- (479,590.86) ;
\draw [shift={(479,655.25)}, rotate = 270] [fill={rgb, 255:red, 0; green, 0; blue, 0 }  ][line width=0.08]  [draw opacity=0] (10.72,-5.15) -- (0,0) -- (10.72,5.15) -- (7.12,0) -- cycle    ;
\draw    (479,582.25) -- (479,515.25) ;
\draw [shift={(479,585.25)}, rotate = 270] [fill={rgb, 255:red, 0; green, 0; blue, 0 }  ][line width=0.08]  [draw opacity=0] (10.72,-5.15) -- (0,0) -- (10.72,5.15) -- (7.12,0) -- cycle    ;
\draw    (479,512.25) -- (479,450.86) ;
\draw [shift={(479,515.25)}, rotate = 270] [fill={rgb, 255:red, 0; green, 0; blue, 0 }  ][line width=0.08]  [draw opacity=0] (10.72,-5.15) -- (0,0) -- (10.72,5.15) -- (7.12,0) -- cycle    ;
\draw    (549,512.25) -- (549,450.86) ;
\draw [shift={(549,515.25)}, rotate = 270] [fill={rgb, 255:red, 0; green, 0; blue, 0 }  ][line width=0.08]  [draw opacity=0] (10.72,-5.15) -- (0,0) -- (10.72,5.15) -- (7.12,0) -- cycle    ;
\draw    (549,582.25) -- (549,520.86) ;
\draw [shift={(549,585.25)}, rotate = 270] [fill={rgb, 255:red, 0; green, 0; blue, 0 }  ][line width=0.08]  [draw opacity=0] (10.72,-5.15) -- (0,0) -- (10.72,5.15) -- (7.12,0) -- cycle    ;
\draw    (549,652.25) -- (549,590.86) ;
\draw [shift={(549,655.25)}, rotate = 270] [fill={rgb, 255:red, 0; green, 0; blue, 0 }  ][line width=0.08]  [draw opacity=0] (10.72,-5.15) -- (0,0) -- (10.72,5.15) -- (7.12,0) -- cycle    ;
\draw    (619,652.25) -- (619,590.86) ;
\draw [shift={(619,655.25)}, rotate = 270] [fill={rgb, 255:red, 0; green, 0; blue, 0 }  ][line width=0.08]  [draw opacity=0] (10.72,-5.15) -- (0,0) -- (10.72,5.15) -- (7.12,0) -- cycle    ;
\draw    (619,582.25) -- (619,520.86) ;
\draw [shift={(619,585.25)}, rotate = 270] [fill={rgb, 255:red, 0; green, 0; blue, 0 }  ][line width=0.08]  [draw opacity=0] (10.72,-5.15) -- (0,0) -- (10.72,5.15) -- (7.12,0) -- cycle    ;
\draw    (619,512.25) -- (619,450.86) ;
\draw [shift={(619,515.25)}, rotate = 270] [fill={rgb, 255:red, 0; green, 0; blue, 0 }  ][line width=0.08]  [draw opacity=0] (10.72,-5.15) -- (0,0) -- (10.72,5.15) -- (7.12,0) -- cycle    ;
\draw    (540.46,450.86) -- (479,450.86) ;
\draw [shift={(543.46,450.86)}, rotate = 180] [fill={rgb, 255:red, 0; green, 0; blue, 0 }  ][line width=0.08]  [draw opacity=0] (10.72,-5.15) -- (0,0) -- (10.72,5.15) -- (7.12,0) -- cycle    ;
\draw    (540.46,660.86) -- (484.54,660.86) ;
\draw [shift={(543.46,660.86)}, rotate = 180] [fill={rgb, 255:red, 0; green, 0; blue, 0 }  ][line width=0.08]  [draw opacity=0] (10.72,-5.15) -- (0,0) -- (10.72,5.15) -- (7.12,0) -- cycle    ;
\draw    (610.46,520.86) -- (549,520.86) ;
\draw [shift={(613.46,520.86)}, rotate = 180] [fill={rgb, 255:red, 0; green, 0; blue, 0 }  ][line width=0.08]  [draw opacity=0] (10.72,-5.15) -- (0,0) -- (10.72,5.15) -- (7.12,0) -- cycle    ;
\draw    (610.46,450.86) -- (549,450.86) ;
\draw [shift={(613.46,450.86)}, rotate = 180] [fill={rgb, 255:red, 0; green, 0; blue, 0 }  ][line width=0.08]  [draw opacity=0] (10.72,-5.15) -- (0,0) -- (10.72,5.15) -- (7.12,0) -- cycle    ;
\draw    (619,520.86) .. controls (532.08,561.74) and (492.98,561.17) .. (411.48,527.5) ;
\draw [shift={(409,526.46)}, rotate = 22.64] [fill={rgb, 255:red, 0; green, 0; blue, 0 }  ][line width=0.08]  [draw opacity=0] (10.72,-5.15) -- (0,0) -- (10.72,5.15) -- (7.12,0) -- cycle    ;
\draw    (619,590.86) .. controls (532.08,631.74) and (492.98,631.17) .. (411.48,597.5) ;
\draw [shift={(409,596.46)}, rotate = 22.64] [fill={rgb, 255:red, 0; green, 0; blue, 0 }  ][line width=0.08]  [draw opacity=0] (10.72,-5.15) -- (0,0) -- (10.72,5.15) -- (7.12,0) -- cycle    ;
\draw    (409,660.86) .. controls (368.11,573.94) and (368.68,534.84) .. (402.36,453.34) ;
\draw [shift={(403.39,450.86)}, rotate = 112.64] [fill={rgb, 255:red, 0; green, 0; blue, 0 }  ][line width=0.08]  [draw opacity=0] (10.72,-5.15) -- (0,0) -- (10.72,5.15) -- (7.12,0) -- cycle    ;
\draw    (479,660.86) .. controls (438.11,573.94) and (438.68,534.84) .. (472.36,453.34) ;
\draw [shift={(473.39,450.86)}, rotate = 112.64] [fill={rgb, 255:red, 0; green, 0; blue, 0 }  ][line width=0.08]  [draw opacity=0] (10.72,-5.15) -- (0,0) -- (10.72,5.15) -- (7.12,0) -- cycle    ;
\draw    (549,660.86) .. controls (508.11,573.94) and (508.68,534.84) .. (542.36,453.34) ;
\draw [shift={(543.39,450.86)}, rotate = 112.64] [fill={rgb, 255:red, 0; green, 0; blue, 0 }  ][line width=0.08]  [draw opacity=0] (10.72,-5.15) -- (0,0) -- (10.72,5.15) -- (7.12,0) -- cycle    ;
\draw    (619.07,660.86) .. controls (578.18,573.94) and (578.75,534.84) .. (612.43,453.34) ;
\draw [shift={(613.46,450.86)}, rotate = 112.64] [fill={rgb, 255:red, 0; green, 0; blue, 0 }  ][line width=0.08]  [draw opacity=0] (10.72,-5.15) -- (0,0) -- (10.72,5.15) -- (7.12,0) -- cycle    ;

\draw  [fill={rgb, 255:red, 0; green, 0; blue, 0 }  ,fill opacity=1 ] (151.74,471.79) .. controls (151.74,468.7) and (154.22,466.19) .. (157.28,466.19) .. controls (160.34,466.19) and (162.82,468.7) .. (162.82,471.79) .. controls (162.82,474.89) and (160.34,477.4) .. (157.28,477.4) .. controls (154.22,477.4) and (151.74,474.89) .. (151.74,471.79) -- cycle ;
\draw  [fill={rgb, 255:red, 0; green, 0; blue, 0 }  ,fill opacity=1 ] (151.74,415.73) .. controls (151.74,412.63) and (154.22,410.12) .. (157.28,410.12) .. controls (160.34,410.12) and (162.82,412.63) .. (162.82,415.73) .. controls (162.82,418.82) and (160.34,421.33) .. (157.28,421.33) .. controls (154.22,421.33) and (151.74,418.82) .. (151.74,415.73) -- cycle ;
\draw    (189.58,527.86) -- (157.28,415.73) ;
\draw    (157.28,471.79) -- (157.28,415.73) ;
\draw  [fill={rgb, 255:red, 0; green, 0; blue, 0 }  ,fill opacity=1 ] (184.04,527.86) .. controls (184.04,524.76) and (186.52,522.25) .. (189.58,522.25) .. controls (192.64,522.25) and (195.12,524.76) .. (195.12,527.86) .. controls (195.12,530.96) and (192.64,533.47) .. (189.58,533.47) .. controls (186.52,533.47) and (184.04,530.96) .. (184.04,527.86) -- cycle ;
\draw  [fill={rgb, 255:red, 0; green, 0; blue, 0 }  ,fill opacity=1 ] (119.45,527.86) .. controls (119.45,524.76) and (121.93,522.25) .. (124.98,522.25) .. controls (128.04,522.25) and (130.52,524.76) .. (130.52,527.86) .. controls (130.52,530.96) and (128.04,533.47) .. (124.98,533.47) .. controls (121.93,533.47) and (119.45,530.96) .. (119.45,527.86) -- cycle ;
\draw  [fill={rgb, 255:red, 0; green, 0; blue, 0 }  ,fill opacity=1 ] (119.45,583.93) .. controls (119.45,580.83) and (121.93,578.32) .. (124.98,578.32) .. controls (128.04,578.32) and (130.52,580.83) .. (130.52,583.93) .. controls (130.52,587.02) and (128.04,589.53) .. (124.98,589.53) .. controls (121.93,589.53) and (119.45,587.02) .. (119.45,583.93) -- cycle ;
\draw  [fill={rgb, 255:red, 0; green, 0; blue, 0 }  ,fill opacity=1 ] (184.04,583.93) .. controls (184.04,580.83) and (186.52,578.32) .. (189.58,578.32) .. controls (192.64,578.32) and (195.12,580.83) .. (195.12,583.93) .. controls (195.12,587.02) and (192.64,589.53) .. (189.58,589.53) .. controls (186.52,589.53) and (184.04,587.02) .. (184.04,583.93) -- cycle ;
\draw  [fill={rgb, 255:red, 0; green, 0; blue, 0 }  ,fill opacity=1 ] (151.74,640.57) .. controls (151.74,637.48) and (154.22,634.97) .. (157.28,634.97) .. controls (160.34,634.97) and (162.82,637.48) .. (162.82,640.57) .. controls (162.82,643.67) and (160.34,646.18) .. (157.28,646.18) .. controls (154.22,646.18) and (151.74,643.67) .. (151.74,640.57) -- cycle ;
\draw  [fill={rgb, 255:red, 0; green, 0; blue, 0 }  ,fill opacity=1 ] (231.99,611.96) .. controls (231.99,608.86) and (234.47,606.35) .. (237.53,606.35) .. controls (240.59,606.35) and (243.07,608.86) .. (243.07,611.96) .. controls (243.07,615.06) and (240.59,617.57) .. (237.53,617.57) .. controls (234.47,617.57) and (231.99,615.06) .. (231.99,611.96) -- cycle ;
\draw  [fill={rgb, 255:red, 0; green, 0; blue, 0 }  ,fill opacity=1 ] (264.29,668.61) .. controls (264.29,665.51) and (266.77,663) .. (269.83,663) .. controls (272.88,663) and (275.36,665.51) .. (275.36,668.61) .. controls (275.36,671.7) and (272.88,674.21) .. (269.83,674.21) .. controls (266.77,674.21) and (264.29,671.7) .. (264.29,668.61) -- cycle ;
\draw  [fill={rgb, 255:red, 0; green, 0; blue, 0 }  ,fill opacity=1 ] (199.69,668.61) .. controls (199.69,665.51) and (202.17,663) .. (205.23,663) .. controls (208.29,663) and (210.77,665.51) .. (210.77,668.61) .. controls (210.77,671.7) and (208.29,674.21) .. (205.23,674.21) .. controls (202.17,674.21) and (199.69,671.7) .. (199.69,668.61) -- cycle ;
\draw  [fill={rgb, 255:red, 0; green, 0; blue, 0 }  ,fill opacity=1 ] (71.5,611.96) .. controls (71.5,608.86) and (73.98,606.35) .. (77.03,606.35) .. controls (80.09,606.35) and (82.57,608.86) .. (82.57,611.96) .. controls (82.57,615.06) and (80.09,617.57) .. (77.03,617.57) .. controls (73.98,617.57) and (71.5,615.06) .. (71.5,611.96) -- cycle ;
\draw  [fill={rgb, 255:red, 0; green, 0; blue, 0 }  ,fill opacity=1 ] (103.8,668.02) .. controls (103.8,664.93) and (106.27,662.42) .. (109.33,662.42) .. controls (112.39,662.42) and (114.87,664.93) .. (114.87,668.02) .. controls (114.87,671.12) and (112.39,673.63) .. (109.33,673.63) .. controls (106.27,673.63) and (103.8,671.12) .. (103.8,668.02) -- cycle ;
\draw  [fill={rgb, 255:red, 0; green, 0; blue, 0 }  ,fill opacity=1 ] (39.2,668.02) .. controls (39.2,664.93) and (41.68,662.42) .. (44.74,662.42) .. controls (47.79,662.42) and (50.27,664.93) .. (50.27,668.02) .. controls (50.27,671.12) and (47.79,673.63) .. (44.74,673.63) .. controls (41.68,673.63) and (39.2,671.12) .. (39.2,668.02) -- cycle ;
\draw  [fill={rgb, 255:red, 0; green, 0; blue, 0 }  ,fill opacity=1 ] (312.24,696.64) .. controls (312.24,693.54) and (314.72,691.03) .. (317.78,691.03) .. controls (320.83,691.03) and (323.31,693.54) .. (323.31,696.64) .. controls (323.31,699.74) and (320.83,702.25) .. (317.78,702.25) .. controls (314.72,702.25) and (312.24,699.74) .. (312.24,696.64) -- cycle ;
\draw  [fill={rgb, 255:red, 0; green, 0; blue, 0 }  ,fill opacity=1 ] (-8.75,696.06) .. controls (-8.75,692.96) and (-6.27,690.45) .. (-3.21,690.45) .. controls (-0.15,690.45) and (2.32,692.96) .. (2.32,696.06) .. controls (2.32,699.15) and (-0.15,701.66) .. (-3.21,701.66) .. controls (-6.27,701.66) and (-8.75,699.15) .. (-8.75,696.06) -- cycle ;
\draw  [fill={rgb, 255:red, 0; green, 0; blue, 0 }  ,fill opacity=1 ] (151.74,605.42) .. controls (151.74,602.32) and (154.22,599.81) .. (157.28,599.81) .. controls (160.34,599.81) and (162.82,602.32) .. (162.82,605.42) .. controls (162.82,608.51) and (160.34,611.02) .. (157.28,611.02) .. controls (154.22,611.02) and (151.74,608.51) .. (151.74,605.42) -- cycle ;
\draw    (124.98,527.86) -- (157.28,415.73) ;
\draw    (124.98,578.32) -- (124.98,533.47) ;
\draw    (189.58,578.32) -- (189.58,533.47) ;
\draw    (157.28,605.42) -- (124.98,583.93) ;
\draw    (157.28,634.97) -- (157.28,611.02) ;
\draw    (157.28,605.42) -- (189.58,583.93) ;
\draw    (80.91,608.91) -- (119.45,583.93) ;
\draw    (109.33,668.02) -- (157.28,640.57) ;
\draw    (189.58,583.93) -- (237.53,611.96) ;
\draw    (162.82,640.57) -- (201.8,664.31) ;
\draw    (-3.21,696.06) -- (109.33,668.02) ;
\draw    (-3.21,696.06) -- (77.03,611.96) ;
\draw    (-3.21,696.06) -- (44.74,668.02) ;
\draw    (269.83,668.61) -- (317.78,696.64) ;
\draw    (205.23,668.61) -- (317.78,696.64) ;
\draw    (237.53,611.96) -- (317.78,696.64) ;

\draw (385.35,430.59) node [anchor=north west][inner sep=0.75pt]  [font=\footnotesize]  {$11$};
\draw (385.35,513.59) node [anchor=north west][inner sep=0.75pt]  [font=\footnotesize]  {$12$};
\draw (385.35,579.59) node [anchor=north west][inner sep=0.75pt]  [font=\footnotesize]  {$13$};
\draw (385.35,672.59) node [anchor=north west][inner sep=0.75pt]  [font=\footnotesize]  {$14$};
\draw (460.18,573.59) node [anchor=north west][inner sep=0.75pt]  [font=\footnotesize]  {$23$};
\draw (465.51,672.59) node [anchor=north west][inner sep=0.75pt]  [font=\footnotesize]  {$24$};
\draw (459.18,504.59) node [anchor=north west][inner sep=0.75pt]  [font=\footnotesize]  {$22$};
\draw (470.18,430.59) node [anchor=north west][inner sep=0.75pt]  [font=\footnotesize]  {$21$};
\draw (538.18,430.59) node [anchor=north west][inner sep=0.75pt]  [font=\footnotesize]  {$31$};
\draw (541.79,672.59) node [anchor=north west][inner sep=0.75pt]  [font=\footnotesize]  {$34$};
\draw (528.18,506.59) node [anchor=north west][inner sep=0.75pt]  [font=\footnotesize]  {$32$};
\draw (528.18,573.59) node [anchor=north west][inner sep=0.75pt]  [font=\footnotesize]  {$33$};
\draw (626,513.59) node [anchor=north west][inner sep=0.75pt]  [font=\footnotesize]  {$42$};
\draw (626,430.59) node [anchor=north west][inner sep=0.75pt]  [font=\footnotesize]  {$41$};
\draw (626,582.59) node [anchor=north west][inner sep=0.75pt]  [font=\footnotesize]  {$43$};
\draw (626,672.59) node [anchor=north west][inner sep=0.75pt]  [font=\footnotesize]  {$44$};
\draw (471.33,705.59) node [anchor=north west][inner sep=0.75pt]    {$H_{2}^{4}(\vec{C}_{4})$};
\draw (150.59,587.36) node [anchor=north west][inner sep=0.75pt]  [font=\footnotesize]  {$11$};
\draw (151.51,649.97) node [anchor=north west][inner sep=0.75pt]  [font=\footnotesize]  {$12$};
\draw (193.03,568.67) node [anchor=north west][inner sep=0.75pt]  [font=\footnotesize]  {$13$};
\draw (106.29,564.93) node [anchor=north west][inner sep=0.75pt]  [font=\footnotesize]  {$14$};
\draw (151.51,396.73) node [anchor=north west][inner sep=0.75pt]  [font=\footnotesize]  {$21$};
\draw (148.74,486.44) node [anchor=north west][inner sep=0.75pt]  [font=\footnotesize]  {$22$};
\draw (197.65,515.41) node [anchor=north west][inner sep=0.75pt]  [font=\footnotesize]  {$23$};
\draw (104.45,514.47) node [anchor=north west][inner sep=0.75pt]  [font=\footnotesize]  {$24$};
\draw (63.84,590.16) node [anchor=north west][inner sep=0.75pt]  [font=\footnotesize]  {$34$};
\draw (53.69,652.77) node [anchor=north west][inner sep=0.75pt]  [font=\footnotesize]  {$33$};
\draw (104.45,676.13) node [anchor=north west][inner sep=0.75pt]  [font=\footnotesize]  {$32$};
\draw (-25.67,691.08) node [anchor=north west][inner sep=0.75pt]  [font=\footnotesize]  {$31$};
\draw (324.58,699.45) node [anchor=north west][inner sep=0.75pt]  [font=\footnotesize]  {$41$};
\draw (196.73,678) node [anchor=north west][inner sep=0.75pt]  [font=\footnotesize]  {$42$};
\draw (242.87,596.7) node [anchor=north west][inner sep=0.75pt]  [font=\footnotesize]  {$43$};
\draw (246.56,656.51) node [anchor=north west][inner sep=0.75pt]  [font=\footnotesize]  {$44$};
\draw (109.32,705.59) node [anchor=north west][inner sep=0.75pt]    {$H_{2}^{4}(\stackrel {\leftrightarrow}{K}_{1,3})$};
\draw (150.59,209.5) node [anchor=north west][inner sep=0.75pt]  [font=\footnotesize]  {$11$};
\draw (151.51,272.11) node [anchor=north west][inner sep=0.75pt]  [font=\footnotesize]  {$12$};
\draw (193.03,190.81) node [anchor=north west][inner sep=0.75pt]  [font=\footnotesize]  {$13$};
\draw (106.29,187.07) node [anchor=north west][inner sep=0.75pt]  [font=\footnotesize]  {$14$};
\draw (151.51,18.88) node [anchor=north west][inner sep=0.75pt]  [font=\footnotesize]  {$21$};
\draw (148.74,108.58) node [anchor=north west][inner sep=0.75pt]  [font=\footnotesize]  {$22$};
\draw (197.65,137.55) node [anchor=north west][inner sep=0.75pt]  [font=\footnotesize]  {$23$};
\draw (104.45,136.61) node [anchor=north west][inner sep=0.75pt]  [font=\footnotesize]  {$24$};
\draw (63.84,212.3) node [anchor=north west][inner sep=0.75pt]  [font=\footnotesize]  {$34$};
\draw (53.69,274.91) node [anchor=north west][inner sep=0.75pt]  [font=\footnotesize]  {$33$};
\draw (104.45,298.27) node [anchor=north west][inner sep=0.75pt]  [font=\footnotesize]  {$32$};
\draw (-25.67,313.22) node [anchor=north west][inner sep=0.75pt]  [font=\footnotesize]  {$31$};
\draw (324.58,321.59) node [anchor=north west][inner sep=0.75pt]  [font=\footnotesize]  {$41$};
\draw (196.73,300.14) node [anchor=north west][inner sep=0.75pt]  [font=\footnotesize]  {$42$};
\draw (242.87,218.85) node [anchor=north west][inner sep=0.75pt]  [font=\footnotesize]  {$43$};
\draw (246.56,278.65) node [anchor=north west][inner sep=0.75pt]  [font=\footnotesize]  {$44$};
\draw (109.32,327.73) node [anchor=north west][inner sep=0.75pt]    {$H_{2}^{4}(\stackrel {\leftrightarrow}{K}_{4})$};
\draw (385.35,56.73) node [anchor=north west][inner sep=0.75pt]  [font=\footnotesize]  {$11$};
\draw (385.35,139.73) node [anchor=north west][inner sep=0.75pt]  [font=\footnotesize]  {$12$};
\draw (385.35,205.73) node [anchor=north west][inner sep=0.75pt]  [font=\footnotesize]  {$13$};
\draw (385.35,298.73) node [anchor=north west][inner sep=0.75pt]  [font=\footnotesize]  {$14$};
\draw (460.18,199.73) node [anchor=north west][inner sep=0.75pt]  [font=\footnotesize]  {$23$};
\draw (465.51,298.73) node [anchor=north west][inner sep=0.75pt]  [font=\footnotesize]  {$24$};
\draw (459.18,130.73) node [anchor=north west][inner sep=0.75pt]  [font=\footnotesize]  {$22$};
\draw (470.18,56.73) node [anchor=north west][inner sep=0.75pt]  [font=\footnotesize]  {$21$};
\draw (538.18,56.73) node [anchor=north west][inner sep=0.75pt]  [font=\footnotesize]  {$31$};
\draw (541.79,298.73) node [anchor=north west][inner sep=0.75pt]  [font=\footnotesize]  {$34$};
\draw (528.18,132.73) node [anchor=north west][inner sep=0.75pt]  [font=\footnotesize]  {$32$};
\draw (528.18,199.73) node [anchor=north west][inner sep=0.75pt]  [font=\footnotesize]  {$33$};
\draw (626,139.73) node [anchor=north west][inner sep=0.75pt]  [font=\footnotesize]  {$42$};
\draw (626,56.73) node [anchor=north west][inner sep=0.75pt]  [font=\footnotesize]  {$41$};
\draw (626,208.73) node [anchor=north west][inner sep=0.75pt]  [font=\footnotesize]  {$43$};
\draw (626,298.73) node [anchor=north west][inner sep=0.75pt]  [font=\footnotesize]  {$44$};
\draw (487.33,327.73) node [anchor=north west][inner sep=0.75pt]    {$H_{2}^{4}(\stackrel {\leftrightarrow}{P}_{4})$};

\end{tikzpicture}

    \caption{The restricted Hanoi graphs of $n=2$ discs and $m=4$ pegs under the restrictions of the movement digraphs of the classical, linear, star, and cyclic variants.}
    \label{fig:my_label2}
\end{figure}
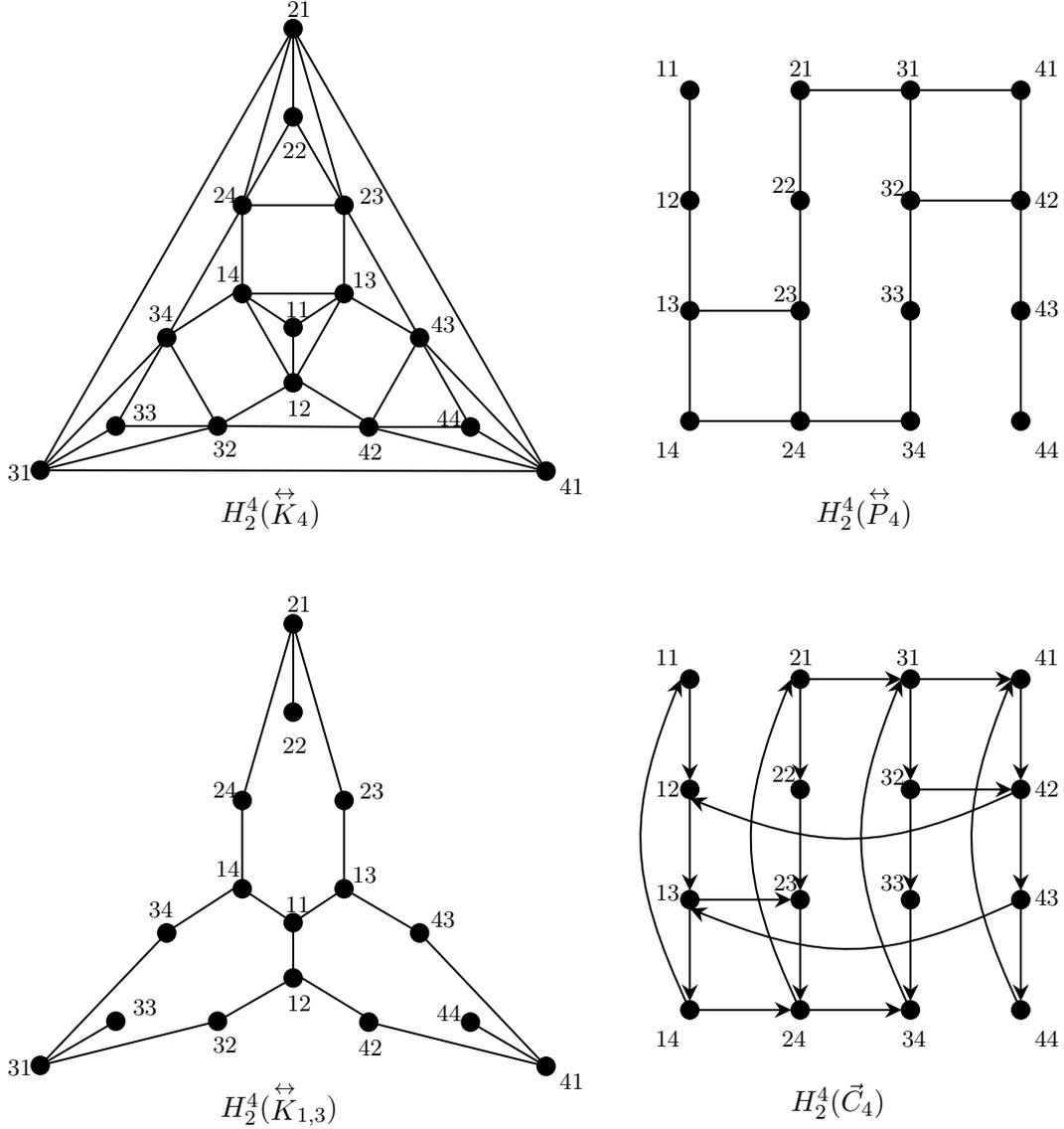

\section{Constructing the restricted Hanoi graphs}
In this section we present a theorem that characterize   the arcs of $H_{n}^{m}(D)$ based on the three rules $(i)-(iii)$ mentioned in the introduction and the properties of the movement digraph. This theorem allows to design a generator of $H_{n}^{m}(D)$ graphs for any number of discs, pegs and any movement digraph $D$.

\begin{theorem}
\label{theoremI} For all integers $n\geq0$, $m\geq 3$, and a digraph $D$ of $m$ vertices, we have  $(u,v)\in A(H_{n}^{m}(D))$ only if there exists a unique  $ k \in \mathcal{N}$ such that the following conditions are fulfilled :
\begin{enumerate}
    \item $u_{k}\neq v_{k}$ and  $u_{d}=v_{d}$ for all $d\in N\setminus\{k\}$; (only one disc can be moved)
    \item $u_{d}\neq u_{k}$  for all $d<k$; (the disc to be moved is a topmost disc) 
    \item  $v_{d}\neq v_{k}$ for all $d<k$; (a disc cannot reside on a smaller one) 
    \item $(u_{k},v_{k})\in A(D)$; (discs are allowed to move from peg $u_{k}$ to peg $v_{k}$)
\end{enumerate}
    
\end{theorem}
\begin{proof}
    The first condition is equivalent to the first rule $(i)$. Therefore,  $(u,v)\in A(H_{n}^{m}(D))$ only if each disc is lying on the same peg in both states $u$ and $v$ except exactly one disc which is equivalent to the existence of a unique index $k\in N$ such that $u_{d}=v_{d}$ for all $d\in N\setminus\{k\}$ and $u_{k}\neq v_{k}$.

    The second condition is equivalent to the second rule $(ii)$, which implies that   $(u,v)\in A(H_{n}^{m}(D))$ requires that all the discs $d<k$ are not lying on the peg on which disc $k$ is lying, namely peg $u_{k}$, which can be insured by implying  $u_{d}=v_{d}\neq u_{k}$ for all $d<k$.

    The third condition is equivalent to the third rule $(iii)$, then $(u,v)\in A(H_{n}^{m}(D))$ only if the smallest disc on the peg on which disc $k$ is going to be moved on, namely peg $v_{k}$, is bigger than disc $k$, in other words, all discs $d<k$ are not lying on peg $v_{k}$. Hence, $u_{d}=v_{d}\neq v_{k}$ for all $d<k$.

    The fourth condition is for respecting the restrictions of the movement digraph $D$. Disc $k$ can be moved from peg $u_{k}$ to peg $v_{k}$ only if the moves from peg $u_{k}$ to peg $v_{k}$ are allowed. Thus, $(u_{k},v_{k})\in A(D)$.

\end{proof}

Theorem \ref{theoremI} allows us to design the simple Algorithm \ref{alg:gen} to generate the restricted Hanoi graphs  $H_{n}^{m}(D)$ for any $n\geq 0$, $m\geq 3$, and digraph of movements $D$ using the boolean function described in Algorithm \ref{alg:cond} that returns true if and only if a candidate arc $(u,v)\in V(H_{n}^{m}(D))$ is admissible to be in  $A(H_{n}^{m}(D))$. 
\begin{algorithm}[H]
		\caption{A boolean function to decide whether an arc is in $H_{n}^{m}(D)$ or not }\label{alg:cond}
		\begin{algorithmic}[1]
			\Function{Condition}{$u,v$}
                \State $\mathcal{K}=\{d\in\mathcal{N}\mid u_{d}\neq v_{d}\};$
                \If{$|\mathcal{K}|\neq 1$}
                \State \Return false;
                \EndIf
                \State let $k$ be the  element of the singleton set $\mathcal{K}$;
                \If{$(u_{k},v_{k})\notin A(D)$}
                \State \Return false;
                \EndIf
                \If{$k\neq 1$}
                \For{$d$ from $k-1$ to $1$}
                \If{$u_{d}=u_{k}$ or $v_{d}=v_{k}$}
                \State \Return false;
                \EndIf
                \EndFor
                \EndIf
                \State \Return true;
			\EndFunction
		\end{algorithmic}
	\end{algorithm}
	\begin{algorithm}[H]
		\caption{The restricted Hanoi graphs generator }\label{alg:gen}
		\begin{algorithmic}[1]
			\Require $n$, $m$ and $D=(\mathcal{M},A(D))$;
                \Ensure The restricted Hanoi graph $H_{n}^{m}(D)$;
                \State generate all vertices of $V(H_{n}^{m}(D))$;
                \State  $A(H_{n}^{m}(D)):=\emptyset$;
                \For{$u\neq v\in V(H_{n}^{m}(D))$}
                \If{$\textsc{Condition}(u,v)=true$} 
                \State $A(H_{n}^{m}(D)):=A(H_{n}^{m}(D))\cup\{(u,v)\}$;
                \EndIf
                \EndFor 
                \State \Return $H_{n}^{m}(D)$;
		\end{algorithmic}
	\end{algorithm}

The first instruction in Algorithm \ref{alg:gen} which is generating the set of vertices, can be performed by any algorithm that generates all the permutations with repetition of the set of pegs $M$. Note that there are $m^{n}$ possible permutations (vertex). 
\begin{proposition}
     For all integers $n\geq0$, $m\geq 3$, and a digraph $D$ of $m$ vertices, the restricted Hanoi graph $H_{n}^{m}(D)$ can be generated  in a time complexity $\mathcal{O}(m^{2n}(m^{2}+n))$.
\end{proposition}
\begin{proof}
    Consider Algorithm \ref{alg:gen} as a generator of  restricted Hanoi graph $H_{n}^{m}(D)$, and let $f(n,m)$ and $g(n,m)$ be the number of operations needed to perform Algorithm \ref{alg:gen} and \ref{alg:cond} respectively, then
    \begin{align*}
        f(n,m)&=m^{n}+m^{n}(m^{n}-1)g(n,m)+1.
    \end{align*}
   In the worst case where $D=\overleftrightarrow{K}_{m}$ and $k\neq 1$, we would have
    \begin{align*}
        g(n,m)&=m^{2}-m+3n-1\\
        &=\mathcal{O}(m^{2}+n).
    \end{align*}
    Thus, 
    \begin{align*}
        \mathcal{O}(f(n,m))&=m^{n}+(m^{2n}-m^{n})\mathcal{O}(m^{2}+n)\\
        &=\mathcal{O}(m^{2n}(m^{2}+n)).
    \end{align*}
    
\end{proof}

For the fourth condition of Theorem \ref{theoremI}, if the movement digraph was that of the classical variant $D=\overleftrightarrow{K}_{m}$ where all the moves between any pair of pegs are allowed, then this condition is not any more important because it is always fulfilled for any values of $n\geq0$, $m\geq 3$, and $k\in \mathcal{N}$. However,  we can reduce the time of verification of the fourth condition depending on the characteristics of the considered movement digraph. The three propositions below replace the fourth condition of Theorem  \ref{theoremI} by a much simpler condition in terms of time verification for the linear, cyclic, and star variants respectively.

\begin{proposition}
    
    In the linear Tower of Hanoi where $D=\overleftrightarrow{P}_{m}$, the fourth condition of Theorem \ref{theoremI} can be replaced by the condition 
    \begin{equation}
        |u_{k}-v_{k}|=1.
    \end{equation}
\end{proposition}
\begin{proof}
    In the linear variant, the only allowed movements are those between adjacent pegs when assuming that pegs are aligned from peg $1$ to peg $m$, which implies that a move of a disc from peg $p$ is either to peg $p+1$ for all $p\in \mathcal{M}\setminus\{m\}$ or to   $p-1$ for all $p\in \mathcal{M}\setminus\{1\}$. Hence, $|p-(p+1)|=|p-(p-1)|=1$ for all $p\in \mathcal{M}$.
\end{proof}

\begin{proposition}
    In the cyclic Tower of Hanoi where $D=\overrightarrow{C}_{m}$, the fourth condition of Theorem \ref{theoremI} can be replaced by the condition 
    \begin{equation}
        v_{k}-u_{k}\in\{1,m-1\}.
    \end{equation}

\end{proposition}
\begin{proof}
    In the cyclic Tower of Hanoi, the only allowed movements are those from peg $p$ to peg $p+1$ for all $p\in  \mathcal{M}\setminus\{m\}$, and from peg $m$ to peg $1$. Therefore, we should have either $v_{k}-u_{k}=p+1-p=1$ or $v_{k}-u_{k}=m-1$. Hence, the result.
\end{proof}
\begin{proposition}
    In the star Tower of Hanoi where $D=\overleftrightarrow{K}_{1,m-1}$, the fourth condition of Theorem \ref{theoremI} can be replaced by the condition 
    \begin{equation}
        |u_{k}-v_{k}|=\max\{u_{k},v_{k}\}-1.
    \end{equation}

\end{proposition}
\begin{proof}
    In the star Tower of Hanoi, the only allowed moves are those from peg $p$ to peg $1$ and from peg $1$ to peg $p$ for all $p\in M\setminus\{1\}$. Thus, to insure that the fourth condition is verified we should have either $u_{k}\neq 1$ and $v_{k}=1$ or $u_{k}=1 $ and $v_{k}\neq 1$, which is equivalent to $|u_{k}-v_{k}|=u_{k}-1$ (the first scenario) or  $|u_{k}-v_{k}|=v_{k}-1$ (the second scenario). Hence, the result.
\end{proof}

\section{Combinatorics on the restricted Hanoi graphs}
In this section we establish some combinatorial results of the restricted Hanoi graphs $H_{n}^{m}(D)$.

Let $t_{k}^{n,m}(D)$ be the number of arcs in the restricted  Hanoi graph $H_{n}^{m}(D)$  that corresponds to the moves of disc $k\in N$, i.e., $t_{k}^{n,m}$ is the number of states that allows disc $k$ to be moved.
\begin{proposition}\label{theo_tknmD}
For all integers $n\geq0$, $m\geq 3$, and a digraph $D$ of $m$ vertices, we have
        \begin{equation}
        t_{k}^{n,m}(D)=|A(D)|m^{n-k}(m-2)^{k-1}.
    \end{equation}
\end{proposition}
\begin{proof}
The arcs to be counted by $t_{k}^{n,m}(D)$  are those that link a state to another by respecting the conditions of Theorem \ref{theoremI}. Therefore, according to condition $1$, $2$ and $3$, the number of possible distributions of discs $n,\ldots,k+1$ (bigger than $k$) on pegs is $m^{n-k}$, and that of  discs $k-1,\ldots,1$ (smaller than $k$) is $(m-2)^{k-1}$. Because, the discs that are bigger than $k$ are allowed to be stacked on any peg among the $m$ pegs, while the discs that are smaller than $k$ are not allowed to be stacked on the peg on which disc $k$ is lying on and the one on which disc $k$ will be moved on. And according to condition $3$, the number of possible pairs of these two last described pegs is exactly the number of arcs in the movement digraph $D$ which is $|A(D)|$, because disc $k$ can be moved from peg $p$ to peg $q$ only if $(p,q)\in A(D)$. Thus, the result is obtained by multiplying the three quantities.

\end{proof}
\begin{corollary}
   For all integers $n\geq0$, $m\geq 3$, we have
   \begin{align}
       t_{k}^{n,m}(\overleftrightarrow{K}_{m})&=(m-1)m^{n-k+1}(m-2)^{k-1},\\
       t_{k}^{n,m}(\overrightarrow{C}_{m})&=m^{n-k+1}(m-2)^{k-1},\\
        t_{k}^{n,m}(\overleftrightarrow{P}_{m})&=t_{k}^{n,m}(\overleftrightarrow{K}_{1,m-1})=2(m-1)m^{n-k}(m-2)^{k-1}.
   \end{align}
\end{corollary}
The next corollary presents two recurrence relations of $t_{k}^{n,m}$ where the first one is in terms of $k$, and the second one is in terms of $n$. 

\begin{corollary}
For all integers $n\geq 0$, $m\geq 3$, and movement digraph $D$,  we have
\begin{align}
    t_{k+1}^{n,m}(D)&=\frac{m-2}{m}t_{k}^{n,m}(D),\\
    t_{k}^{n+1,m}(D)&=m t_{k}^{n,m}(D).
\end{align}
\end{corollary}

Using $t_{k}^{n,m}(D)$, it will be easy to calculate the number of arcs in $H_{n}^{m}(D)$, which is the result of the next theorem. Not that the number of arcs in $H_{n}^{m}(D)$ where $D$ is an undirected connected graph has been given in \cite{hinz2022dudeney} and proved using induction proof based on a  recurrence relation for the graphs  $H_{n}^{m}(D)$  with an undirected connected movement graph $D$.

\begin{theorem}
        For all integers $n\geq 0$, $m\geq 3$, and movement digraph $D$, the number of arcs in $H_{n}^{m}(D)$ is 
    \begin{equation}
        |A(H_{n}^{m}(D))|= \sum\limits_{k=1}^{n}t_{k}^{n,m}(D)=\frac{1}{2}|A(D)|(m^{n}-(m-2)^{n}).
    \end{equation}
\end{theorem}
\begin{corollary}
   For all integers $n\geq0$, $m\geq 3$, we have
\begin{align}
    |A(H_{n}^{m}(\overleftrightarrow{K}_{m}))|&=\frac{1}{2}m(m-1)(m^{n}-(m-2)^{n}),\\
    |A(H_{n}^{m}(\overrightarrow{C}_{m}))|&=\frac{1}{2}m(m^{n}-(m-2)^{n}),\\
    |A(H_{n}^{m}(\overleftrightarrow{P}_{m}))|&=|A(H_{n}^{m}(\overleftrightarrow{K}_{1,m-1}))|=(m-1)(m^{n}-(m-2)^{n}).
\end{align}
\end{corollary}

Let $u=u_{n}\ldots u_{1}$ be a vertex of $H_{n}^{m}(D)$, then for all $(p,q)\in D$,  $\mathcal{P}_{u}^{p}=\{i\mid u_{i}=p\}$ and $\mathcal{Q}_{u}^{q}=\{i\mid u_{i}=q\}$ represents the set of discs lying on peg $p$ and $q$ respectively in accordance with distribution of discs of the state $u$. Then the outdegree of vertex $u$ can be calculated  using the following simple algorithm.

\begin{algorithm}[H]
		\caption{A function to calculate the outdegree of vertices of $H_{n}^{m}(D)$ }
		\begin{algorithmic}[1]
			\Function{Degree}{$u$}
                \State $d^{+}(u):=0$;
                \For{$(p,q)\in A(D)$}
                \If{$\mathcal{P}_{u}^{p}\neq \emptyset$}
                \If{$\mathcal{Q}_{u}^{q}=\emptyset$ or $\min \mathcal{P}_{u}^{p} < \min \mathcal{Q}_{u}^{q} $ }
                \State $d^{+}(u):=d^{+}(u)+1$;
                \EndIf
			\EndIf
                \EndFor
                \State \Return $d^{+}(u)$.
			\EndFunction
		\end{algorithmic}
	\end{algorithm}

\begin{theorem}
For all integers $n\geq 0$, $m\geq 3$, and movement digraph $D$, the degree of $u\in V(H_{n}^{m}(D))$ can be calculated using the following expression.
    \begin{equation}
        d^{+}(u)=\sum\limits_{(p,q)\in A(D)} x_{\mathcal{P}_{u}^{p}}(y_{\mathcal{P}_{u}^{p},\mathcal{Q}_{u}^{q}}-x_{\mathcal{Q}_{u}^{q}}+1).
    \end{equation}
    Where $x_{A}$ and $y_{A,B}$ are two indicator functions defined as follows, 
    \begin{equation*}x_{A}=
        \begin{cases}
            1,&\text{if $A\neq \emptyset$};\\
            0,&\text{otherwise.}
        \end{cases}\quad\quad \quad
        y_{A,B}=
        \begin{cases}
            1,&\text{if $x_{A}=x_{B}=1$ and $\min A\leq \min B$};\\
            0,&\text{otherwise.}
        \end{cases}
    \end{equation*}
\end{theorem}
With $A,B\subseteq \mathcal{N}$.
\begin{proof}
If $\mathcal{P}_{u}^{p}=\emptyset$ then $x_{\mathcal{P}_{u}^{p}}=0$ 
 which gives  $x_{\mathcal{P}_{u}^{p}}(y_{\mathcal{P}_{u}^{p},\mathcal{Q}_{u}^{q}}-x_{\mathcal{Q}_{u}^{q}}+1)=0$ which is the case where peg $p$ is empty, so no disc can be moved from peg $p$ to another peg. Otherwise, if $\mathcal{Q}_{u}^{q}=\emptyset$  then $x_{\mathcal{Q}_{u}^{q}}=0$  and $y_{\mathcal{P}_{u}^{p},\mathcal{Q}_{u}^{q}}=0$ which implies that 
    $x_{\mathcal{P}_{u}^{p}}(y_{\mathcal{P}_{u}^{p},\mathcal{Q}_{u}^{q}}-x_{\mathcal{Q}_{u}^{q}}+1)=1$. And for the case where  $\mathcal{Q}_{u}^{q}\neq \emptyset$, we have $x_{\mathcal{Q}_{u}^{q}}=1$ and $y_{\mathcal{P}_{u}^{p},\mathcal{Q}_{u}^{q}}=1$ only if $\min \mathcal{P}_{u}^{p} < \min \mathcal{Q}_{u}^{q} $ implying   $x_{\mathcal{P}_{u}^{p}}(y_{\mathcal{P}_{u}^{p},\mathcal{Q}_{u}^{q}}-x_{\mathcal{Q}_{u}^{q}}+1)=1$.
\end{proof}

Now we can  deduce a new identity for the number of arcs in $(H_{n}^{m}(D))$.

\begin{corollary}
For all integers $n\geq 0$, $m\geq 3$, and movement digraph $D$, we have

    \begin{equation}
        |A(H_{n}^{m}(D))|=\sum\limits_{k=1}^{n}t_{k}^{n,m}(D)=\frac{1}{2}|A(D)|(m^{n}-(m-2)^{n})=\sum\limits_{  u\in V(H_{n}^{m}(D))}\sum\limits_{(p,q)\in A(D)} x_{\mathcal{P}_{u}^{p}}(y_{\mathcal{P}_{u}^{p},\mathcal{Q}_{u}^{q}}-x_{\mathcal{Q}_{u}^{q}}+1).
    \end{equation}
    
\end{corollary}

\footnotesize

\bibliographystyle{plain}
\bibliography{BibTHL.bib}

\begin{thebibliography}{10}

\bibitem{allouche2005restricted}
Jean-Paul Allouche and Amir Sapir.
\newblock Restricted towers of hanoi and morphisms.
\newblock In {\em Developments in Language Theory: 9th International
  Conference, DLT 2005, Palermo, Italy, July 4-8, 2005. Proceedings 9}, pages
  1--10. Springer, 2005.

\bibitem{atkinson1981cyclic}
MD~Atkinson.
\newblock The cyclic towers of hanoi.
\newblock {\em Inf. Process. Lett.}, 13(3):118--119, 1981.

\bibitem{azriel2008infinite}
Dany Azriel, Noam Solomon, and Shay Solomon.
\newblock On an infinite family of solvable hanoi graphs.
\newblock {\em ACM Transactions on Algorithms (TALG)}, 5(1):1--22, 2008.

\bibitem{berend2006cyclic}
Daniel Berend and Amir Sapir.
\newblock The cyclic multi-peg tower of hanoi.
\newblock {\em ACM Transactions on Algorithms (TALG)}, 2(3):297--317, 2006.

\bibitem{berend2006diameter}
Daniel Berend and Amir Sapir.
\newblock The diameter of hanoi graphs.
\newblock {\em Information processing letters}, 98(2):79--85, 2006.

\bibitem{berend2012tower}
Daniel Berend, Amir Sapir, and Shay Solomon.
\newblock The tower of hanoi problem on pathh graphs.
\newblock {\em Discrete Applied Mathematics}, 160(10-11):1465--1483, 2012.

\bibitem{dudeney2002canterbury}
Henry~Ernest Dudeney.
\newblock {\em The Canterbury Puzzles}.
\newblock Courier Corporation, 2002.

\bibitem{frame1941solution}
James~Sutherland Frame.
\newblock Solution to advanced problem 3918.
\newblock {\em Amer. Math. Monthly}, 48:216--217, 1941.

\bibitem{hinz2018tower}
Andreas~M Hinz, Sandi Klavar, and Ciril Petr.
\newblock The tower of hanoi myths and maths, 2018.

\bibitem{hinz2022dudeney}
Andreas~M Hinz, Borut Lu{\v{z}}ar, and Ciril Petr.
\newblock The dudeney--stockmeyer conjecture.
\newblock {\em Discrete Applied Mathematics}, 319:19--26, 2022.

\bibitem{itard1960ed}
Jean Itard.
\newblock Ed. lucas, r{\'e}cr{\'e}ations math{\'e}matiques.
\newblock {\em Revue d'histoire des sciences}, 13(3):272--272, 1960.

\bibitem{klavvzar2005hanoi}
Sandi Klav{\v{z}}ar, Uro{\v{s}} Milutinovi{\'c}, and Ciril Petr.
\newblock Hanoi graphs and some classical numbers.
\newblock {\em Expositiones Mathematicae}, 23(4):371--378, 2005.

\bibitem{leiss2013worst}
Ernst~L Leiss.
\newblock The worst hanoi graphs.
\newblock {\em Theoretical Computer Science}, 498:100--106, 2013.

\bibitem{sapir2004tower}
Amir Sapir.
\newblock The tower of hanoi with forbidden moves.
\newblock {\em Comput. J.}, 47(1):20--24, 2004.

\bibitem{stockmeyer1994variations}
Paul~K Stockmeyer.
\newblock Variations on the four-post tower of hanoi puzzle.
\newblock In {\em Congress. Numer.}, volume 102, pages 3--12, 1994.

\bibitem{wood1980towers}
Derick Wood.
\newblock {\em The towers of Brahma and Hanoi revisited}.
\newblock Unit for computer Science, McMaster University, 1980.

\end{thebibliography}

\end{document}